\def\myfontsize{10pt}
\IfFileExists{comrich.conf}{\input{comrich.conf}}{}
\documentclass[\myfontsize]{amsart}
\usepackage{amsmath, amsfonts, amssymb}
\usepackage{stmaryrd}
\usepackage{accents}
\usepackage[all]{xy}
\usepackage{graphicx}
\usepackage{hhline}
\usepackage{standalone}
\usepackage{tikz}
\usetikzlibrary{arrows.meta}
\usepackage{tableau}
\usetikzlibrary{patterns,calc}
\usepackage[colorlinks=true, linkcolor=blue,
  citecolor=blue, urlcolor=blue]{hyperref}
\ifcsname mypreamble\endcsname\mypreamble\fi

\def\newthm#1#2{\newtheorem{#1}[dummy]{#2}%
  \expandafter\def\csname#2\endcsname##1{\hyperref[#1:##1]{#2~\ref*{#1:##1}}}}
\newthm{thm}{Theorem}
\newthm{lemma}{Lemma}
\newthm{prop}{Proposition}
\newthm{cor}{Corollary}
\newthm{conj}{Conjecture}
\newthm{cons}{Consequence}
\newthm{claim}{Claim}
\newthm{fact}{Fact}
\newthm{obs}{Observation}
\newthm{guess}{Guess}
\theoremstyle{definition}
\newthm{defn}{Definition}
\newthm{remark}{Remark}
\newthm{example}{Example}
\newthm{question}{Question}
\newthm{notation}{Notation}
\newthm{problem}{Problem}

\newcommand{\Section}[1]{\hyperref[sec:#1]{Section~\ref*{sec:#1}}}
\newcommand{\Table}[1]{\hyperref[tab:#1]{Table~\ref*{tab:#1}}}
\newcommand{\Figure}[1]{\hyperref[fig:#1]{Figure~\ref*{fig:#1}}}
\newcommand{\eqn}[1]{\hyperref[eqn:#1]{(\ref*{eqn:#1})}}

\DeclareMathOperator{\GL}{GL}
\DeclareMathOperator{\SL}{SL}

\DeclareMathOperator{\Sp}{Sp}
\DeclareMathOperator{\SO}{SO}
\DeclareMathOperator{\Gr}{Gr}
\DeclareMathOperator{\Fl}{Fl}
\DeclareMathOperator{\LG}{LG}

\DeclareMathOperator{\SG}{SG}
\DeclareMathOperator{\SF}{SF}
\DeclareMathOperator{\OG}{OG}
\DeclareMathOperator{\OF}{OF}

\DeclareMathOperator{\Hom}{Hom}

\DeclareMathOperator{\Span}{Span}
\DeclareMathOperator{\Spec}{Spec}

\DeclareMathOperator{\chr}{char}
\DeclareMathOperator{\codim}{codim}

\DeclareMathOperator{\Pic}{Pic}
\newcommand{\Cl}{\operatorname{C\ell}}

\DeclareMathOperator{\lf}{lf}
\DeclareMathOperator{\rt}{rt}
\DeclareMathOperator{\up}{up}
\DeclareMathOperator{\dn}{dn}

\DeclareMathOperator{\ur}{ur}

\newcommand{\ssm}{\smallsetminus}
\newcommand{\bA}{{\mathbb A}}

\newcommand{\bG}{{\mathbb G}}
\newcommand{\bK}{{\mathbb K}}
\newcommand{\bP}{{\mathbb P}}
\newcommand{\C}{{\mathbb C}}
\newcommand{\R}{{\mathbb R}}
\newcommand{\Q}{{\mathbb Q}}
\newcommand{\Z}{{\mathbb Z}}

\newcommand{\cH}{{\mathcal H}}

\def\cL{{\mathcal L}}

\newcommand{\cO}{{\mathcal O}}
\newcommand{\cP}{{\mathcal P}}

\newcommand{\cV}{{\mathcal V}}

\newcommand\oX{\accentset{\circ}{X}}

\newcommand{\bull}{{\scriptscriptstyle \bullet}}
\newcommand{\al}{{\alpha}}
\newcommand{\be}{{\beta}}
\newcommand{\ga}{{\gamma}}
\newcommand{\de}{{\delta}}
\newcommand{\ep}{{\epsilon}}

\newcommand{\ka}{{\kappa}}
\newcommand{\la}{{\lambda}}

\newcommand{\om}{{\omega}}

\newcommand{\dex}{\delta_X}

\newcommand{\wt}{\widetilde}
\newcommand{\wh}{\widehat}

\newcommand{\ov}{\overline}
\newcommand{\pic}[2]{\includegraphics[scale=#1]{Figures/#2}}
\newcommand{\ignore}[1]{}

\newcommand{\noin}{\noindent}

\newcommand{\titem}[1]{\noin{\rm(#1)}}

\begin{document}

\title{Cominuscule Richardson varieties}

\date{October 13, 2026}

\author{Anders~Skovsted~Buch}
\address{Department of Mathematics, Rutgers University, 110
  Frelinghuysen Road, Piscataway, NJ 08854, USA}
\email{asbuch@math.rutgers.edu}


% \keywords{}

\thanks{The author was partially supported by NSF grant DMS-2152316.}

\begin{abstract}
  A Richardson variety in a cominuscule flag variety is defined by a skew shape
  of boxes. We prove that any cominuscule Richardson variety has a unique
  decomposition as a product of indecomposable cominuscule Richardson varieties,
  which are given by the connected components of the skew shape. We also prove
  that the first Chow group and the Picard group of a cominuscule Richardson
  variety are dual to each other; both are free abelian groups of rank equal to
  the number of indecomposable factors. Along the way, we construct an explicit
  desingularization of any cominuscule Richardson variety of classical Lie type
  as a tower of (isotropic) Grassmannian bundles. Our results imply that the
  Hilbert polynomial of the minimal ample line bundle is a well-defined
  geometric invariant of any cominuscule Richardson variety. We give a positive
  formula for this Hilbert polynomial.
\end{abstract}

\maketitle

\setcounter{tocdepth}{1}
\tableofcontents

% !TeX root=comrich.tex

\section{Introduction}

Let $X = G/P_X$ be a flag variety defined by a reductive linear algebraic group
$G$ and a parabolic subgroup $P_X$. Fix a maximal torus $T$ and a Borel subgroup
$B$ such that $T \subset B \subset P_X \subset G$. Each Weyl group element $u
\in W$ defines the Schubert varieties $X_u = \ov{B u.P_X}$ and $X^u = \ov{B^-
u.P_X}$. A non-empty intersection $X_u^v = X_u \cap X^v$ of opposite Schubert
varieties is called a \emph{Richardson variety}. In this paper we study these
Richardson varieties when $X$ is a cominuscule flag variety. We will show that
any cominuscule Richardson variety $X_u^v$ has a unique factorization as a
product of \emph{indecomposable} cominuscule Richardson varieties. In addition,
we prove that the first Chow group $A_1(X_u^v)$ and the Picard group
$\Pic(X_u^v)$ are dual to each other; both are free abelian groups of rank equal
to the number of indecomposable factors of $X_u^v$. The family of cominuscule
flag varieties consists of Grassmannians of type A, Lagrangian Grassmannians,
maximal orthogonal Grassmannians, quadric hypersurfaces, and two exceptional
varieties of Lie type E called the Cayley Plane $E_6/P_6$ and the Freudenthal
variety $E_7/P_7$ (see \Table{labsets}). More generally, any product of these
varieties will be considered a cominuscule flag variety. Our computation of the
Picard groups of cominuscule Richardson varieties of type E uses a new result of
Anderson and Payne \cite{anderson.payne:chow}.

Some properties of Richardson varieties can be proved from the corresponding
properties of Schubert varieties. In particular, classical results show that
Schubert varieties are normal, Cohen-Macaulay, and have rational singularities
\cite{mehta.ramanathan:frobenius, ramanan.ramanathan:projective,
ramanathan:schubert}. These properties are inherited by Richardson varieties,
thanks to a generalization of Kleiman's transversality theorem
\cite{kleiman:transversality, brion:positivity, knutson.lam.ea:projections}.
Both Schubert varieties and Richardson varieties are rational varieties.
Moreover, any Schubert variety $X_u$ is stratified by the $B$-orbits it
contains, each of which is isomorphic to an affine space. Each Chow group
$A_d(X_u)$ is therefore a free abelian group generated by the classes of all
$d$-dimensional Schubert varieties contained in $X_u$. The rationality of
Richardson varieties follows from Deodhar's decomposition \cite{deodhar:some},
which expresses any Richardson variety as a disjoint union of locally closed
cells of the form $\bA^n \times (\bA^1 \ssm \{0\})^m$. However, Deodhar's
decomposition is not a stratification, and the Chow groups of Richardson
varieties are only known in special cases (see, e.g.,\
\cite{gonzales.le.ea:quasisymmetric} for the case of toric Richardson varieties
in Grassmannians).

The Picard group $\Pic(X_u)$ of any Schubert variety was obtained by Mathieu
\cite{mathieu:formules*1}. The usual strategy for computing the Picard group of
a normal variety $Y$ is first to compute the divisor class group $\Cl(Y) =
A_{\dim(Y)-1}(Y)$ and then identify $\Pic(Y)$ as a subgroup of $\Cl(Y)$.
However, aside from toric Richardson varieties and Schubert varieties, the
divisor class group $\Cl(X_u^v)$ appears to be unknown. The fact that $X_u^v$ is
rational implies that $\Cl(X_u^v)$ is finitely generated. The only other result
in this direction that we are aware of is that, when $X = \Gr(m,n)$ is a
Grassmannian of type A, the class group $\Cl(\oX_u^v)$ of any open Richardson
variety $\oX_u^v = Bu.P_X \cap B^-v.P_X$ is zero \cite{levinson.purbhoo:class}.

Given any complete variety $Y$, the natural pairing $\Pic(Y) \otimes A_1(Y) \to
\Z$ given by $L \otimes [C] \mapsto \deg(L|_C)$ defines a group homomorphism
$\Pic(Y) \to A_1(Y)^*$ that we call the \emph{degree map} of $Y$. Here
$A_1(Y)^* = \Hom_\Z(A_1(Y),\Z)$ denotes the dual of the first Chow group. Rather
than relying on the class group, we show that the degree map $\Pic(X_u^v)
\to A_1(X_u^v)^*$ is an isomorphism for any cominuscule Richardson variety
$X_u^v$. We note that Mathieu obtained the Picard group $\Pic(X_u)$ of a
Schubert variety as a subgroup of the class group $\Cl(X_u)$, but he also proved
\cite[Prop.~6]{mathieu:formules*1} that the restriction map $\Pic(X) \to
\Pic(X_u)$ is surjective and has the same kernel as the degree map $\Pic(X)
\to A_1(X_u)^*$; this implies $\Pic(X_u) \cong A_1(X_u)^*$, as the degree
map of a Schubert variety is clearly surjective. While this isomorphism
generalizes to cominuscule Richardson varieties, the restriction map $\Pic(X)
\to \Pic(X_u^v)$ generally fails to be surjective for Richardson varieties.

The flag variety $X = G/P_X$ is called \emph{cominuscule} if $P_X$ is a maximal
parabolic subgroup defined by a cominuscule simple root $\ga$, that is, the
coefficient of $\ga$ in the highest root is 1. We assume this in the following.
The combinatorics of $X$ is then centered around the partially ordered set
$\cP_X = \{ \al \in \Phi \mid \al \geq \ga \}$, where $\Phi$ is the root system,
and we write $\al \geq \be$ if $\al-\be$ is a sum of simple roots (see
\Table{labsets}). In particular, a result of Proctor \cite{proctor:bruhat} shows
that the subset $W^X \subset W$ of minimal length representatives is in
bijective correspondence with the set of lower order ideals in $\cP_X$; this
bijection sends any element $u \in W^X$ to its inversion set $I(u) = \{\al \in
\Phi^+ \mid u.\al \in \Phi^-\}$. The lower order ideals in $\cP_X$ generalize
the (shifted) Young diagrams used for classical Grassmannians. For this reason,
the order ideals will be called \emph{straight shapes} and the roots in $\cP_X$
will be called \emph{boxes}. Given a straight shape $\la \subset \cP_X$, we
let $w_\la \in W^X$ denote the element with $I(w_\la) = \la$, and write $X_\la =
\ov{B w_\la.P_X}$ and $X^\la = \ov{B^- w_\la.P_X}$ for the associated Schubert
varieties. We then have $\dim(X_\la) = \codim(X^\la,X) = |\la|$, and the Bruhat
order on $X$ is given by
\[
  \mu \subset \la \ \Leftrightarrow \
  X_\mu \subset X_\la \ \Leftrightarrow \
  X^\mu \supset X^\la \ \Leftrightarrow \
  X_\la \cap X^\mu \neq \emptyset \,.
\]
Given straight shapes $\mu \subset \la \subset \cP_X$, the associated
\emph{cominuscule Richardson variety} is denoted by $X_\la^\mu = X_\la \cap
X^\mu$. This variety depends only on the skew shape $\theta = \la/\mu := \la
\ssm \mu$ up to translation; we let $X_\theta$ denote the isomorphism class of
$X_\la^\mu$ as an algebraic variety.

An irreducible variety of positive dimension will be called
\emph{indecomposable} if it is not a product of varieties of strictly smaller
dimension. Our main results are the following theorems.

\begin{thm}\label{thm:intro-indecomp}%
  Let $X$ be a cominuscule flag variety and $\theta \subset \cP_X$ a skew shape.
  Then the following are equivalent. {\rm(a)} $\theta$ is connected (as a
  poset). {\rm(b)} $X_\theta$ is indecomposable. {\rm(c)} $A_1(X_\theta) \cong
  \Z$. {\rm(d)} $\Pic(X_\theta) \cong \Z$.
\end{thm}

\begin{thm}\label{thm:intro-chowpic}%
  Let $X$ be a cominuscule flag variety, let $\theta \subset \cP_X$ be a skew
  shape, and let $\theta_1, \dots, \theta_r$ be the connected components of
  $\theta$.\smallskip

  \noin{\rm(a)} We have $X_\theta \cong X_{\theta_1} \times \dots \times
  X_{\theta_r}$.\smallskip

  \noin{\rm(b)} The map $A_1(X_\theta) \to A_1(X_{\theta_1}) \oplus \dots \oplus
  A_1(X_{\theta_r})$ defined by proper pushforward to the factors is an
  isomorphism of groups.\smallskip

  \noin{\rm(c)} The map $\Pic(X_{\theta_1}) \oplus \dots \oplus
  \Pic(X_{\theta_r}) \to \Pic(X_\theta)$ defined by pullback from the factors is
  an isomorphism of groups.\smallskip

  \noin{\rm(d)} The degree map $\Pic(X_\theta) \to A_1(X_\theta)^*$ is an
  isomorphism of groups.
\end{thm}

In particular, both $A_1(X_\theta)$ and $\Pic(X_\theta)$ are free abelian groups
of rank equal to the number of indecomposable factors of $X_\theta$. Part (a) of
\Theorem{intro-chowpic} was proved for cominuscule Grassmannians of classical
Lie types in \cite{buch.ravikumar:pieri}. We obtain it as a special case of a
factorization result for Richardson varieties in arbitrary flag varieties (see
\Proposition{flagvar-decomp}).

In a paper in preparation with Ayah Almousa and Daoji Huang
\cite{almousa.buch.ea:gorenstein}, we have generalized the above results to all
positroid varieties and used this to determine which positroid varieties are
Gorenstein.

Given a cominuscule Richardson variety $X_\theta \subset X$, we define the
\emph{minimal embedding} of $X_\theta$ to be $L.X_\theta$, where $L \subset G$
is the minimal Levi subgroup such that $X_\theta$ is contained in a single
$L$-orbit. We show that this minimal embedding is itself a Richardson variety in
$X$, and also a cominuscule flag variety of $L$; if $X_\theta$ has several
indecomposable factors, then $L.X_\theta$ will be a product of several
indecomposable cominuscule flag varieties.

\begin{thm}\label{thm:intro-minembed}%
  Let $i: X_\theta \subset X$ be a cominuscule Richardson variety such that $X$
  is the minimal embedding of $X_\theta$. Then the pushforward map $i_* :
  A_1(X_\theta) \to A_1(X)$ and the pullback map $i^* : \Pic(X) \to
  \Pic(X_\theta)$ are isomorphisms.
\end{thm}

When $\theta \subset \cP_X$ is any connected skew shape, the pushforward
$A_1(X_\theta) \to A_1(X)$ and pullback $\Pic(X) \to \Pic(X_\theta)$ are always
injective. Both maps are isomorphisms if and only if $\theta$ contains a long
box (or the root system $\Phi$ is simply-laced); otherwise the images of the
maps are $2 A_1(X)$ and $2 \Pic(X_\theta)$, respectively.

A result of Hirschowitz \cite{hirschowitz:groupe} implies that the first Chow
group $A_1(X_\theta)$ is generated by the $T$-stable curves contained in
$X_\theta$. To compute the group $A_1(X_\theta)$, we must account for the
relations. We will say that two $T$-fixed points in $X$ are \emph{adjacent} if
they are joined by a $T$-stable curve; in this case the $T$-stable curve is
unique, and the points are comparable in the Bruhat order of $X^T$. Our
computation of $A_1(X_\theta)$ is based on the following result. Part (a) holds
for all flag varieties of type different from $G_2$, whereas part (b) is true
only if $X$ is cominuscule or has type $G_2$ and Picard rank 1.

\begin{prop}\label{prop:intro-triangles}%
  Let $X$ be a cominuscule flag variety.\smallskip

  \noin{\rm(a)} Assume that $x,y,z \in X^T$ are pairwise adjacent $T$-fixed
  points. Then there exists a unique surface $\Omega \subset X$ with fixed locus
  $\Omega^T = \{x,y,z\}$. Moreover, $\Omega$ is a translate of a Schubert
  variety.\smallskip

  \noin{\rm(b)} Let $\tau \subset \sigma \subset \cP_X$ be straight shapes such
  that $w_\tau.P_X$ and $w_\sigma.P_X$ are adjacent $T$-fixed points, and assume
  that $|\sigma/\tau| \geq 2$. Then there exists a straight shape $\ka$ such
  that $\tau \subsetneq \ka \subsetneq \sigma$ and $w_\ka.P_X$ is adjacent to
  both $w_\tau.P_X$ and $w_\sigma.P_X$.
\end{prop}

We also use the fact that any Richardson variety $X_\theta \subset X$ is
\emph{$T$-convex} in $X$ \cite{buch.chaput.ea:equivariant}, that is, if $\Omega
\subset X$ is any $T$-stable closed subvariety such that $\Omega^T \subset
X_\theta$, then $\Omega \subset X_\theta$. Part (a) of the proposition therefore
implies that, if three $T$-stable curves contained in $X_\theta$ form a
triangle, then they define the same class in $A_1(X_\theta)$, up to a constant.
Indeed, the surface $\Omega$ defined by the vertices is contained in $X_\theta$
since $X_\theta$ is $T$-convex, and the three curves are contained in $\Omega$
since $\Omega$ is itself $T$-convex. Using this, part (b) implies that
$A_1(X_\theta)$ is generated by the classes of all Richardson curves in
$X_\theta$. When $\theta$ is connected, it follows from similar arguments that
these Richardson curves generate a subgroup of rank 1 in $A_1(X_\theta)$. Then
the pushforward map $A_1(X_\theta) \to A_1(X)$ is injective, and in fact an
isomorphism when $\theta$ contains a long box. The surjectivity of the
degree map $\Pic(X_\theta) \to A_1(X_\theta)^*$ also follows, using that
the degree map of $X$ is surjective.

To prove that the degree map $\Pic(X_\theta) \to A_1(X_\theta)^*$ is
injective, we use the existence of a desingularization $\pi : \wt{X_\theta} \to
X_\theta$ such that the degree map of $\wt{X_\theta}$ is injective. In
fact, these degree maps fit in a commutative diagram
\[
  \xymatrix{
    \Pic(X_\theta) \ar[r]^{\pi^*} \ar[d] &
    \Pic(\wt{X_\theta}) \ar[d]
    \\
    A_1(X_\theta)^* \ar[r]^{(\pi_*)^*} &
    A_1(\wt{X_\theta})^* \,,
  }
\]
and the normality of $X_\theta$ implies that $\pi^*$ is injective. For
cominuscule Richardson varieties of classical types, we explicitly construct the
desingularization $\wt{X_\theta}$ as a tower of (isotropic) Grassmannian
bundles, in the style of Zelevinsky's desingularization of Schubert varieties in
Grassmannians of type A \cite{zelevinski:small}. This desingularization is of
independent interest, as it is simpler and more explicit than the usual `brick'
resolution studied in \cite{brion:positivity, escobar:brick}. In response to our
questions, David Anderson and Sam Payne have proved that the brick resolution
indeed satisfies the required injectivity property (and much more!). We rely on
this result to complete the computation of the Picard groups of cominuscule
Richardson varieties of type E. We also use this result for Richardson varieties
in the maximal orthogonal Grassmannian $\OG(n,2n)$ when the ground field has
characteristic 2.

One source of motivation for this paper is the problem of classifying
cominuscule Richardson varieties up to isomorphism as algebraic varieties.
Richmond and Slofstra have previously obtained such a classification of the
Schubert varieties in full flag varieties $G/B$
\cite{richmond.slofstra:isomorphism}. The classification of the Schubert
varieties in Grassmannians of type A was obtained in
\cite{tarigradschi.xu:isomorphism} and generalized to all cominuscule Schubert
varieties in \cite{richmond.tarigradschi.ea:isomorphism*1}.

Define a \emph{long/short poset} to be a partially ordered set $\cP$ together
with a partition $\cP = \cP_{\textrm{long}} \coprod \cP_{\textrm{short}}$ into
disjoint subsets of long and short boxes, such that every connected component of
$\cP$ has non-empty intersection with $\cP_{\textrm{long}}$. It was proved in
\cite{richmond.tarigradschi.ea:isomorphism*1} that two cominuscule Schubert
varieties $X_\la$ and $X'_{\la'}$ are isomorphic if and only if the associated
straight shapes $\la\subset \cP_X$ and $\la' \subset \cP_{X'}$ are isomorphic as
long/short posets. Knowledge of the Picard group $\Pic(X_\la)$ was important in
this work; the Schubert classes in the total Chow group $A_*(X_\la)$ can be
identified as the generators of the cone of effective classes, after which the
action of the Chern class $c_1(\cL)$ of the ample generator of $\Pic(X_\la)$
reveals the Bruhat order, and $\la$ is recovered as the subset of
join-irreducible elements in the Bruhat order.

It follows from \Theorem{intro-chowpic} that two cominuscule Richardson
varieties are isomorphic if and only if their indecomposable factors are
isomorphic. This reduces the classification problem to that of classifying the
indecomposable cominuscule Richardson varieties. The following conjecture was
posed in \cite{tarigradschi.xu:isomorphism} for Richardson varieties in
Grassmannians of type A.

\begin{conj}\label{conj:intro-classify}%
  Let $X$ and $X'$ be cominuscule flag varieties and let $\theta \subset \cP_X$
  and $\theta' \subset \cP_{X'}$ be connected skew shapes. Then the Richardson
  varieties $X_\theta$ and $X'_{\theta'}$ are isomorphic as varieties if and
  only if $\theta$ and $\theta'$ are isomorphic or anti-isomorphic as long/short
  posets.
\end{conj}

An \emph{anti-isomorphism} of long/short posets reverses the order but preserves
the lengths of boxes. In a sequel \cite{buch:towards} to this paper we will
prove one direction: if $\theta$ and $\theta'$ are isomorphic or anti-isomorphic
as long/short posets, then $X_\theta$ and $X'_{\theta'}$ are isomorphic as
varieties. We have also verified by computer that \Conjecture{intro-classify} is
true when $\theta$ and $\theta'$ have at most 15 boxes.

The determination of the Picard group of a cominuscule Richardson variety
$X_\theta$ makes it possible to define related geometric invariants that may be
helpful for the classification problem. One example is the Hilbert polynomial
defined by the minimal ample line bundle in $\Pic(X_\theta)$; this line bundle
exists and is unique by \Theorem{intro-chowpic}. In \Section{hilbert} we give a
simple recursive formula for this Hilbert polynomial, which shows that it has a
positive expansion in the binomial basis.

This paper is organized as follows. In \Section{flagvar} we fix our notation for
flag varieties and recall some standard definitions. We also recall the
construction of flag varieties of classical Lie types as varieties of
(isotropic) flags in a fixed vector space. In \Section{decomp} we define flag
subvarieties and minimal embeddings of Richardson varieties in arbitrary flag
varieties, and we prove a decomposition theorem that expresses certain
Richardson varieties as products of smaller Richardson varieties.
\Section{comin} interprets these results for cominuscule flag varieties. We also
prove that the minimal embedding of a cominuscule Richardson variety is a
cominuscule flag variety. \Section{chowpic} determines the first Chow group and
Picard group of a cominuscule Richardson variety. \Section{desing} constructs a
desingularization of any cominuscule Richardson variety of classical Lie type as
a tower of (isotropic) Grassmannian bundles. Finally, \Section{hilbert} computes
the Hilbert polynomial defined by the minimal ample line bundle on any
cominuscule Richardson variety.

We have benefited from helpful comments, references, and discussions with
several people, especially David Anderson, Pierre-Emmanuel Chaput, William
Fulton, William Graham, Sam Payne, Nicolas Perrin, and Weihong Xu.

% !TeX root=comrich.tex

\section{Flag varieties}\label{sec:flagvar}

In this section we fix our notation for flag varieties and recall some results.

\subsection{General notation}

Let $G$ be a connected reductive linear algebraic group over an algebraically
closed field $\bK$, and choose a maximal torus $T$ and a Borel subgroup $B$ such
that $T \subset B \subset G$. The opposite Borel subgroup $B^- \subset G$ is
defined by $B \cap B^- = T$. Let $\Phi$ be the root system of $(G,T)$, with
positive roots $\Phi^+$ and simple roots $\Delta \subset \Phi^+$ given by $B$.
We will not assume that $\Phi$ is irreducible. We identify $\Delta$ with the set
of nodes in the Dynkin diagram of $\Phi$. Let $W = N_G(T)/T$ be the Weyl group
of $G$. The reflection along a root $\al \in \Phi$ is denoted $s_\al$. The
inversion set of $u \in W$ is $I(u) = \{ \al \in \Phi^+ \mid u.\al < 0 \}$. We
have $\ell(u) = \# I(u)$.

A complete homogeneous $G$-variety $X$ will be called a \emph{flag variety} of
$G$ if the morphism $G \to X$, $g \mapsto g.x$ defined by any point $x \in X$ is
separable. Any such flag variety $X$ contains a unique $B$-fixed point. We
denote the parabolic subgroup stabilizing this point by $P_X \subset G$ and the
point itself by $1.P_X$. Each element $g \in G$ defines a point $g.P_X =
g.(1.P_X)$ in $X$. We identify $X$ with the quotient $G/P_X = \{ g.P_X \mid g
\in G \}$. Let $W_X = N_{P_X}(T)/T$ be the Weyl group of $P_X$, and let $W^X
\subset W$ be the subset of minimal representatives of the cosets in $W/W_X$.
Each element $u \in W$ defines a $T$-fixed point $u.P_X \in X^T$ and the
\emph{Schubert varieties} $X_u = \ov{Bu.P_X}$ and $X^u = \ov{B^-u.P_X}$. These
Schubert varieties depend only on the coset $u W_X$ in $W/W_X$, and for $u \in
W^X$ we have $\dim(X_u) = \codim(X^u,X) = \ell(u)$. Let $\leq$ denote the Bruhat
order on $W$. For $u, v \in W^X$ we have
\[
  v \leq u
  \ \ \Leftrightarrow \ \
  X_v \subset X_u
  \ \ \Leftrightarrow \ \
  X^v \supset X^u
  \ \ \Leftrightarrow \ \
  X_u \cap X^v \neq \emptyset \,.
\]
In this case the intersection $X^v_u = X_u \cap X^v$ is called a
\emph{Richardson variety}; this variety is reduced, irreducible, rational, and
$\dim(X_u^v) = \ell(u)-\ell(v)$; see \cite{kleiman:transversality, deodhar:some}
and \cite[\S2]{brion.kumar:frobenius}. In addition, $X_u^v$ is normal and has
rational singularities \cite{mehta.ramanathan:frobenius,
ramanan.ramanathan:projective, ramanathan:schubert, brion:positivity,
knutson.lam.ea:projections}. It was proved in
\cite[Thm.~6.3]{buch.chaput.ea:equivariant} that $X_u^v$ is \emph{$T$-convex} in
$X$, that is, if $Z \subset X$ is any $T$-stable closed subvariety such that the
fixed point set $Z^T$ is contained in $X_u^v$, then $Z \subset X_u^v$. Unless
stated otherwise, a Richardson variety in $X$ will always be defined relative to
the chosen Borel subgroups $B$ and $B^-$.

Each element $u \in W$ has a unique factorization $u = u^X u_X$ for which $u^X
\in W^X$ and $u_X \in W_X$, called the \emph{parabolic factorization} of $u$
with respect to $W_X$. This factorization is \emph{reduced} in the sense that
$\ell(u) = \ell(u^X) + \ell(u_X)$. The parabolic factorization of the longest
element $w_0 \in W$ is denoted $w_0 = w_0^X w_{0,X}$, where $w_0^X$ is the
longest element in $W^X$ and $w_{0,X}$ is the longest element in $W_X$. Notice
that $u \in W^X$ if and only if $u < u s_\be$ holds for all $\be \in \Delta$
with $s_\be \in W_X$.

\subsection{Flag varieties of classical Lie types}\label{sec:classical}

In the rest of this section we fix our notation for flag varieties of the
classical Lie types $A_n$, $B_n$, $C_n$, and $D_n$. This notation will be used
in examples as well as in some type-dependent proofs.

The root system $\Phi$ is identified with a subset of $\R^{n'}$, where $n' =
n+1$ in type $A_n$, and $n'=n$ in the other classical types. The basis of simple
roots will be denoted $\Delta = \{\be_1, \dots, \be_n\}$ with the standard
numbering (see \cite[\S11.4]{humphreys:introduction*1} or \Table{labsets}). Let
$\{\ep_1, \dots, \ep_{n'}\}$ be the coordinate basis of $\R^{n'}$. We set $\be_i
= \ep_i-\ep_{i+1}$ for $1 \leq i \leq n-1$. The last simple root $\be_n$ depends
on the type and is given in the following table together with $\Phi$. The root
systems $B_2$ and $C_2$ are isomorphic, as are $A_3$ and $D_3$.
\smallskip

\newcommand{\myspca}{\vspace{-2.5ex}&&&\\}
\newcommand{\myspcb}{\vspace{-2.5ex}&&&\\}
\newcommand{\myspcc}{\vspace{-2.5ex}\\&&&\\}
\begin{center}
\begin{tabular}{|c|c|l|l|}
\hline\myspca
Type & Rank & Root system $\Phi$ & Last simple root \\
\myspcb\hline\myspca
$A_n$ & $n \geq 1$ & $\{ \ep_i-\ep_j \mid i\neq j \in [1,n+1] \}$ &
$\be_n = \ep_n-\ep_{n+1}$ \\
\myspcb\hline\myspca
$B_n$ & $n \geq 2$ &
$\{ \ep_i-\ep_j,\, \ep_i+\ep_j,\, \ep_i \mid i \neq j \in [1,n] \}$ &
$\be_n=\ep_n$ \\
\myspcb\hline\myspca
$C_n$ & $n \geq 2$ &
$\{ \ep_i-\ep_j,\, \ep_i+\ep_j,\, 2\ep_i \mid i\neq j \in [1,n] \}$ &
$\be_n=2\ep_n$ \\
\myspcb\hline\myspca
$D_n$ & $n \geq 3$ &
$\{ \ep_i-\ep_j,\, \ep_i+\ep_j \mid i\neq j \in [1,n] \}$ &
$\be_n=\ep_{n-1}+\ep_n$ \myspcc\hline
\end{tabular}
\end{center}
\smallskip

The Weyl group $W$ is identified with a subgroup of the permutation group $S_N$,
where $N$ depends on the Lie type and is given by the table below. When $\Phi$
has type $A_n$, we have $W = S_N$. When $\Phi$ has type $B_n$, $C_n$, or $D_n$,
it is convenient to write $\ov{i} = N+1-i$ for $i \in [1,N]$. In type $A_n$ we
set $\ov{i} = i$ for $i \in [1,N]$. Then $W \subset S_N$ is the subgroup of
permutations satisfying the condition in the following table.\smallskip

\renewcommand{\myspca}{\vspace{-2.5ex}&&\\}%
\renewcommand{\myspcb}{\vspace{-2.5ex}&&\\}%
\renewcommand{\myspcc}{\vspace{-2.5ex}\\&&\\}%
\begin{center}
\begin{tabular}{|c|c|l|}
\hline\myspca
Type & $N$ & $W$ is the set of permutations $w \in S_N$ satisfying \\
\myspcb\hline\myspca
$A_n$ & $n+1$ & $w \in S_N$ \\
\myspcb\hline\myspca
$B_n$ & $2n\!+\!1$ & $w(\ov{\,i\,}) = \ov{w(i)}$ for $i \in [1,N]$ \\
\myspcb\hline\myspca
$C_n$ & $2n$ & $w(\ov{\,i\,}) = \ov{w(i)}$ for $i \in [1,N]$ \\
\myspcb\hline\myspca
$D_n$ & $2n$ & $w(\ov{\,i\,}) = \ov{w(i)}$ for $i \in [1,N]$\, and\,
$\# \{ i \in [1,n] : w(i) > n \}$ is even
\myspcc\hline
\end{tabular}
\end{center}
\smallskip

For $i, j \in [1,N]$, we let $\tau_{i,j} \in S_N$ be the permutation defined by
$\tau_{i,j}(i)=j$, $\tau_{i,j}(j)=i$, $\tau_{i,j}(\ov{i})=\ov{j}$,
$\tau_{i,j}(\ov{j})=\ov{i}$, and $\tau_{i,j}(k)=k$ for $k \in
[1,N]\ssm\{i,j,\ov{i},\ov{j}\}$. The reflection $s_\al \in W$ defined by each
root $\al \in \Phi$ is given by the following table. The simple reflections are
denoted $s_i = s_{\be_i}$ for $i \in [1,n]$.\smallskip

\renewcommand{\myspca}{\vspace{-2.5ex}&&&\\}
\renewcommand{\myspcb}{\vspace{-2.5ex}&&&\\}
\renewcommand{\myspcc}{\vspace{-2.5ex}\\&&&\\}
\begin{center}
\begin{tabular}{|l|c|c|c|}
\hline\myspca
Root $\al \in \Phi$ & $\ep_i-\ep_j$ & $\ep_i+\ep_j$ & $\ep_i$, $2\ep_i$ \\
\myspcb\hline\myspca
Reflection $s_\al \in W$ & $\tau_{i,j}$ & $\tau_{i,\ov{j}}$ & $\tau_{i,\ov{i}}$
\myspcc\hline
\end{tabular}
\end{center}
\smallskip

\noin
For $w \in W$ and $1 \leq i \leq n$, we have $w < w s_i$ in the Bruhat order if
and only if $w(i) < w(i+1)$, with one exception: If $\Phi$ has type $D_n$ and
$i=n$, then $w < w s_n$ holds if and only if $w(n-1) < w(n+1)$.

Any element $w \in W$ can be represented in \emph{one-line notation} by the
sequence of numbers $w(1)\, w(2)\, \cdots w(n')$, where `bar' notation is used
for integers larger than $n'$. For example, $w = 4\,\ov{1}\,2\,\ov{5}\,\ov{3}$
means $w(1)=4$, $w(2)=\ov{1}$, $w(3)=2$, $w(4)=\ov{5}$, and $w(5)=\ov{3}$.
Notice that since the number of bars is odd, this element belongs to $W$ only
when $\Phi$ has type $B_5$ or $C_5$. Notice also that $\ov{3}=11+1-3=9$ when
$\Phi$ has type $B_5$, whereas $\ov{3}=10+1-3=8$ when $\Phi$ has type $C_5$.

Flag varieties of the classical Lie types will be constructed as spaces of
isotropic flags in $\bK^N$. The standard construction used here requires that
the field $\bK$ does not have characteristic 2 when $\Phi$ has type $B_n$ or
$D_n$. We first define a bilinear form $\om$ on $\bK^N$ which is zero in type
$A_n$, is non-degenerate and symmetric in types $B_n$ and $D_n$, and is
non-degenerate and anti-symmetric in type $C_n$. Let $\{e_1, \dots, e_N\}$ be
the coordinate basis of $\bK^N$. When $\om$ is not zero, it is determined by
$\om(e_i,e_j) = \de_{i+j,N+1}$ for $1 \leq i \leq j \leq N$. Let $G \subset
\SL(N)$ be the subgroup preserving $\om$, that is
\[
  G = \{ g \in \SL(N) \mid \om(g.x, g.y) = \om(x,y)
  ~\forall\, x,y \in \bK^N \}\,.
\]
This group is $\SL(N)$ in type $A_n$, is the special orthogonal group $\SO(N)$
in types $B_n$ and $D_n$, and is the symplectic group $\Sp(N)$ in type $C_n$.
Let $T \subset G$ be the maximal torus of all diagonal matrices in $G$, let $B
\subset G$ be the Borel subgroup of upper triangular matrices, and let $B^-
\subset G$ be the lower triangular matrices.

A vector subspace $V \subset \bK^N$ is called \emph{isotropic} if the bilinear
form $\om$ vanishes on $V$, that is $\om(V,V)=0$. This condition implies
$\dim(V) \leq n'$. All subspaces $V \subset \bK^N$ are isotropic when $\Phi$ has
type $A_n$. Given a subset $a = \{a_1 < a_2 < \dots < a_r\}$ of $[0,n'] =
\{0,1,\dots,n'\}$, an $a$-dimensional isotropic \emph{flag} $V_\bull \subset
\bK^N$ is an increasing sequence $V_\bull = (V_{a_1} \subset V_{a_2} \subset
\dots \subset V_{a_r})$ of isotropic subspaces of $\bK^N$ with $\dim(V_{a_i}) =
a_i$ for $1 \leq i \leq r$. Note that the steps of a flag are indexed by their
dimensions. Let $F_\om(a;N) = F_\om(a_1,\dots,a_r;N)$ denote (one connected
component of) the variety of all $a$-dimensional isotropic flags in $\bK^N$.
Then $F_\om(a;N)$ is a flag variety of $G$. Given vectors $v_1, \dots, v_\ell
\in \bK^N$, we will use the notation $\langle v_1, \dots, v_\ell\rangle =
\Span_\bK\{v_1, \dots, v_\ell\}$. The $B$-stable point of $X = F_\om(a;N)$ is
then given by
\[
  1.P_X = (\langle e_1,\dots,e_{a_1}\rangle \subset
  \langle e_1,\dots,e_{a_2}\rangle \subset \dots \subset
  \langle e_1,\dots,e_{a_r}\rangle) \,.
\]
More generally, $V_\bull \in F_\om(a;N)$ is a $T$-fixed point if and only if
each subspace $V_{a_i}$ is spanned by a subset of the basis $\{e_1,\dots,e_N\}$.
The variety of all isotropic flags of dimension $a$ is connected unless $\Phi$
has type $D_n$ and $a_r=n$; in this case there are two components, and we let
$F_\om(a;N)$ be the component containing $1.P_X$. In type $D_n$ we also have
$F_\om(a;N) \cong F_\om(a \cup \{n\};N)$ whenever $n-1 \in a$, so we will only
consider dimension sets $a = \{a_1,\dots,a_r\}$ satisfying $n-1 \notin a$ or $n
\in a$. When this convention is followed, the parabolic subgroup $P_X$ defining
$X = F_\om(a;N)$ is determined in all Lie types by $s_i \in W^X$ if and only if
$i \in \{a_1,\dots,a_r\}$.

Let each permutation $w \in S_N$ act on $\bK^N$ by $w(e_i) = e_{w(i)}$ for $i
\in [1,N]$; this identifies $S_N$ with the group of permutation matrices in
$\GL(N)$. The action of $W$ on the fixed point locus $F_\om(a;N)^T$ is given by
$w.V_\bull = (w.V_{a_1} \subset \dots \subset w.V_{a_r})$. The permutation
matrix corresponding to $w \in W$ belongs to $G$ if and only if $\det(w)=1$.
When $\det(w)=-1$, we obtain an element $w' \in G$ by switching the sign of any
entry of the permutation matrix defined by $w$, in which case $w$ is represented
by $w'$ as an element of $N_G(T)/T$.

Let $F_\om(N) = F_\om(1,2,\dots,n;N)$ denote the variety of complete isotropic
flags in $\bK^N$. The space $F_\om(m;N)$, for $m \in [1,n]$, is the Grassmannian
of $m$-dimensional isotropic subspaces $V \subset \bK^N$. The following notation
will be used to refer to flag varieties of all classical types.\smallskip

\renewcommand{\myspca}{\vspace{-2.5ex}&&&&&\\}
\renewcommand{\myspcb}{\vspace{-2.5ex}&&&&&\\}
\renewcommand{\myspcc}{\vspace{-2.5ex}\\&&&&&\\}
\begin{center}
\begin{tabular}{|c|c|c|c|c|c|}
\hline\myspca
Type & $N$ & $\om$ & $F_\om(a;N)$ & $F_\om(N)$ & $F_\om(m;N)$ \\
\myspcb\hline\myspca
$A_n$ & $n+1$ & zero & $\Fl(a;N)$ & $\Fl(N)$ & $\Gr(m,N)$ \\
\myspcb\hline\myspca
$B_n$ & $2n+1$ & orthogonal & $\OF(a;N)$ & $\OF(N)$ & $\OG(m,N)$ \\
\myspcb\hline\myspca
$C_n$ & $2n$ & symplectic & $\SF(a;N)$ & $\SF(N)$ & $\SG(m,N)$ \\
\myspcb\hline\myspca
$D_n$ & $2n$ & orthogonal & $\OF(a;N)$ & $\OF(N)$ & $\OG(m,N)$
\myspcc\hline
\end{tabular}
\end{center}
\smallskip

The symplectic Grassmannian $\SG(n,2n)$ of maximal isotropic subspaces in
$\bK^{2n}$ is called the \emph{Lagrangian Grassmannian} and is also denoted
$\LG(n,2n)$. The \emph{quadric hypersurface} of dimension $d\geq 3$ is denoted
$Q^d = \OG(1,d+2)$. We have $Q^3 \cong \LG(2,4)$ and $Q^4 \cong \Gr(2,4)$. In
addition, the maximal orthogonal Grassmannians $\OG(n,2n)$ and $\OG(n-1,2n-1)$,
of types $D_n$ and $B_{n-1}$, are isomorphic by the map $f : \OG(n,2n) \cong
\OG(n-1,2n-1)$ defined by $f(V) = \iota^{-1}(V)$, where $\iota : \bK^{2n-1} \to
\bK^{2n}$ is the linear map given by $\iota(e_i)=e_i$ for $0 < i < n$,
$\iota(e_i)=e_{i+1}$ for $n<i<2n$, and $\iota(e_n)=\sqrt{1/2}\,(e_n+e_{n+1})$.

% !TeX root=comrich.tex

\section{Product decomposition of Richardson varieties}\label{sec:decomp}

In this section we let $X = G/P_X$ be an arbitrary flag variety. We will show
that certain Richardson varieties in $X$ are products of smaller Richardson
varieties, generalizing the decomposition result in \cite{buch.ravikumar:pieri}
for cominuscule Richardson varieties of classical types.

\subsection{The poset of a flag variety}

The root lattice $\Span_\Z(\Delta)$ has a partial order defined by $\al' \leq
\al$ if and only if $\al-\al'$ is a sum of positive roots. Define the partially
ordered set
\[
  \cP_X \,=\, I(w_0^X) \,=\, \Phi^+ \ssm I(w_{0,X}) \,=\,
  \{ \al \in \Phi^+ \mid w_0^X.\al < 0 \} \,,
\]
where the partial order is inherited from $\Span_\Z(\Delta)$. Equivalently,
$\cP_X$ is the set of roots $\al \in \Phi$ for which $-\al$ is a $T$-weight of
the tangent space $T_X(1.P_X)$. The cardinality of $\cP_X$ is $\#\cP_X =
\dim(X)$, and the minimal elements of $\cP_X$ are the excluded simple roots of
$P_X$,
\[
  \min(\cP_X) \,=\, \cP_X \cap \Delta \,=\,
  \{ \be \in \Delta \mid s_\be \in W^X \} \,.
\]
The corresponding Schubert divisor classes $\{ [X^{s_\be}] \mid \be \in
\min(\cP_X) \}$ form a $\Z$-basis of the Picard group $\Pic(X)$. The identity
$\cP_X = \Phi^+ \ssm I(w_{0,X})$ implies that
\begin{equation}\label{eqn:min2PX}%
  \cP_X \,=\,
  \{\al \in \Phi \mid \exists\, \be \in \min(\cP_X) : \be \leq \al\} \,.
\end{equation}

\begin{remark}
  One can check that $\al_1 \leq \al_2$ is a covering relation in $\cP_X$ if and
  only if $\al_2 - \al_1 \in \Delta$ is a simple root. We will not need this
  fact.
\end{remark}

\subsection{Flag subvarieties}\label{sec:flagsubvar}%

Given a subset $S \subset \Delta$ of the simple roots, let $G_S \subset G$ be
the Levi subgroup containing $T$ with root system $\Phi_S = \Span_\Z(S) \cap
\Phi$, see \cite[(8.4.1)]{springer:linear*1}. Our chosen Borel subgroups of
$G_S$ are $B_S = G_S \cap B$ and $B_S^- = G_S \cap B^-$, and the positive roots
defined by $B_S$ are $\Phi_S^+ = \Phi_S \cap \Phi^+$. The Weyl group of $G_S$ is
the subgroup $W_S \subset W$ generated by the simple reflections $\{s_\be \mid
\be \in S\}$.

\begin{defn}\label{defn:flagsubvar}%
  Let $\ka.P_X \in X$ be a $T$-fixed point and let $S \subset \Delta$. The
  $G_S$-orbit $G_S \ka.P_X \subset X$ will be called a \emph{flag subvariety} of
  $X$.
\end{defn}

This definition of a flag subvariety depends on the group $G$ acting on $X$, as
well as the chosen maximal torus $T$ and Borel subgroup $B$. Our definition
implies a close relationship between Richardson varieties in $X$ and Richardson
varieties in flag subvarieties of $X$, which suits the goals of this paper, see
\Proposition{flagsubrich}.

Any flag subvariety $Y = G_S \ka.P_X \subset X$ is a flag variety of $G_S$ by
\Proposition{flagsubvar} below. The $B_S$-fixed point in $Y$ is denoted $1.P_Y$
and has stabilizer $P_Y \subset G_S$. Let $(W_S)^Y \subset W_S$ be the subset of
minimal representatives modulo $(W_S)_Y = N_{P_Y}(T)/T$. Let $w_0^Y \in (W_S)^Y$
denote the longest element, and let $\cP_Y = I(w_0^Y) \subset \Phi_S^+$ be the
associated partially ordered set. For $\ka \in W^X$, let $\ka_S \in W^X$ denote
the unique minimal element of the set $\{w \ka \mid w \in W_S\}$. Equivalently,
$\ka_S^{-1}$ is the left factor of the parabolic factorization of $\ka^{-1}$
with respect to $W_S$. Part (a) of the following result implies that the
inclusion $Y \to X$ is given by $g.P_Y \mapsto g \ka_S.P_X$ for $g \in G_S$.

\begin{prop}\label{prop:flagsubvar}%
  Let $X = G/P_X$ be a flag variety and let $\ka \in W^X$. Let $Y = G_S \ka.P_X
  \subset X$ be the flag subvariety defined by $S \subset \Delta$ and
  $\ka.P_X$.\smallskip

  \titem{a} $Y$ is a flag variety of $G_S$ with $B_S$-fixed point $1.P_Y =
  \ka_S.P_X$.\smallskip

  \titem{b} $Y = X_{w_0^Y \ka_S}^{\ka_S}$ is a Richardson variety in
  $X$.\smallskip

  \titem{c} The map $w \mapsto w \ka_S$ is an order isomorphism from the Bruhat
  order of $(W_S)^Y$ to the interval $[\ka_S, w_0^Y \ka_S]$ in the Bruhat order
  of $W^X$.\smallskip

  \titem{d} The map $\al \mapsto \ka_S^{-1}.\al$ is an order-preserving
  bijection $\cP_Y \to I(w_0^Y \ka_S) \ssm I(\ka_S)$.
\end{prop}
\begin{proof}
  Since $\ka.P_X$ is stabilized by the Borel subgroup $G_S \cap \ka B \ka^{-1}$
  of $G_S$ \cite[(6.4.7)]{springer:linear*1}, $Y$ is a flag variety of $G_S$.
  The $T$-fixed points are $Y^T = \{ w \ka.P_X \mid w \in W_S \}$. Since
  $\ov{B_S \ka_S.P_X}^T \subset X_{\ka_S} \cap Y^T = \{ \ka_S.P_X \}$, it
  follows that $\ka_S.P_X$ is a $B_S$-fixed point of $Y$, which proves part (a).
  The definition of $\ka_S$ implies that the product $w \ka_S$ is reduced for
  each $w \in W_S$, so the map $w \mapsto w \ka_S$ is an order isomorphism $W_S
  \to W_S \ka_S$. The element $w \in W_S$ belongs to $(W_S)^Y$ if and only if
  $w$ is the minimal representative of the point $w.P_Y = w \ka_S.P_X$. Using
  that $\ka_S \in W^X$, we deduce that $w \in (W_S)^Y$ if and only if $w \ka_S
  \in W^X$. This proves part (c), which in turn implies $Y^T \subset X_{w_0^Y
  \ka_S}^{\ka_S}$ and $\dim(Y) = \dim(X_{w_0^Y \ka_S}^{\ka_S})$. Part (b)
  follows from this, using that Richardson varieties are $T$-convex
  \cite{buch.chaput.ea:equivariant}. Finally, part (d) holds because
  $\ka_S^{-1}.\al > 0$ for all $\al \in \Phi_S^+$ by the definition of $\ka_S$.
\end{proof}

\begin{remark}
  The set $-\cP_Y = \{-\al \mid \al \in \cP_Y\}$ consists of all $T$-weights of
  the tangent space $T_Y(1.P_Y) \subset T_X(\ka_S.P_X)$, while
  $-\ka_S^{-1}.\cP_Y$ is the set of $T$-weights of the translate
  $\ka_S^{-1}.T_Y(1.P_Y) \subset T_X(1.P_X)$. This explains the geometric
  relevance of \Proposition{flagsubvar}(d).
\end{remark}

We will show in \Proposition{comin} that when $X$ is a cominuscule flag variety,
all flag subvarieties of $X$ are also cominuscule. The following example shows
that the general situation is more complicated.

\begin{example}\label{example:fsvSG26}%
  Let $X = \SG(2,6)$ be the symplectic Grassmannian of 2-dimensional isotropic
  subspaces of $\bK^6$, see \Section{classical}. Let $S = \{\be_1, \be_2\}
  \subset \Delta$ be the two short simple roots, so that $G_S$ has type $A_2$,
  and set $\ka = \ka_S = s_3 s_2$ and $Y = G_S \ka.P_X$. Then $\ka_S.P_X = s_3
  s_2.\langle e_1, e_2 \rangle = \langle e_1, e_4 \rangle$, and the map $Y \to
  \Fl(3)$ defined by $V \mapsto (V \cap \langle e_1,e_2,e_3\rangle, V^\perp \cap
  \langle e_1,e_2,e_3\rangle)$ is an isomorphism. The list of flag subvarieties
  of $X$ includes $\bP^1$, $\bP^1\times\bP^1$, $\bP^2$, $\bP^3$, $\Fl(3)$,
  $\LG(2,4)$, $\SG(2,6)$, and the single point $\Spec(\bK)$. Some of these
  varieties have several inclusions as flag subvarieties of $X$.
\end{example}

When $X$ is a cominuscule flag variety, it follows from
\cite[Lemma~4.2(a)]{buch.chaput.ea:positivity} that the map of
\Proposition{flagsubvar}(d) is an order isomorphism, in the sense that the
inverse map $\ka_S : I(w_0^Y \ka_S) \ssm I(\ka_S) \to \cP_Y$ is also
order-preserving. The following example shows that this is not the case for
arbitrary flag varieties.

\begin{example}
  Let $X = \Fl(1,2;4)$ be the variety of two-step flags in $\bK^4$ with
  dimension $(1,2)$, and set $S = \{\be_1, \be_3\}$, $\ka = \ka_S = s_2 s_1$,
  and $Y = G_S \ka.P_X$. We obtain $\ka_S.P_X = (\langle e_3 \rangle, \langle
  e_1, e_3 \rangle)$, $Y = \bP^1 \times \bP^1$, $w_0^Y = s_1 s_3$, $\cP_Y = S$,
  $\ka_S^{-1}.\be_1 = \be_2$, and $\ka_S^{-1}.\be_3 = \be_1+\be_2+\be_3$. The
  two roots in $\ka_S^{-1}.\cP_Y = I(w_0^Y \ka_S) \ssm I(\ka_S)$ are comparable,
  while the roots in $\cP_Y$ are not.
\end{example}

\subsection{Minimal embeddings}

We next show that each Richardson variety $X_u^v$ is contained in a unique
minimal flag subvariety of $X$. The Schubert varieties $Y_{u'} = \ov{B_S
u'.P_Y}$ and $Y^{v'} = \ov{B_S^- v'.P_Y}$ in the flag subvariety $Y = G_S
\ka.P_X \subset X$, defined for $v', u' \in W_S$, may not be Schubert varieties
in $X$, but any Richardson variety $Y_{u'}^{v'} = Y_{u'} \cap Y^{v'}$ is also a
Richardson variety in $X$ by the following result.

\begin{prop}\label{prop:flagsubrich}%
  Let $X = G/P_X$ be a flag variety, let $u, v, \ka \in W^X$ satisfy $v \leq u$,
  and let $Y = G_S \ka.P_X$ be the flag subvariety of $X$ defined by $\ka.P_X$
  and $S \subset \Delta$.\smallskip

  \titem{a} The intersection $X_u^v \cap Y$ is either empty or a Richardson
  variety in $Y$.\smallskip

  \titem{b} We have $X_u^v \subset Y$ if and only if $u \ka^{-1} \in W_S$
  and $v \ka^{-1} \in W_S$.\smallskip

  \titem{c} Given $v' \leq u'$ in $(W_S)^Y$, we have $Y_{u'}^{v'} = X_{u'
  \ka_S}^{v' \ka_S}$.
\end{prop}
\begin{proof}
  Assume $X^v \cap Y \neq \emptyset$. Using that all Richardson varieties in $X$
  are compatibly Frobenius split \cite[\S2]{brion.kumar:frobenius}, it follows
  from \Proposition{flagsubvar}(b) that $X^v \cap Y$ is reduced. Since $X^v \cap
  Y$ is also $B_S^-$-stable, it is a union of $B_S^-$-stable Schubert varieties
  in $Y$. The $T$-fixed points are $(X^v \cap Y)^T = \{ w \ka_S.P_X \mid w \in
  W_S \text{ and } w \ka_S \geq v \}$. Since the set $\{ w \in W_S \mid w \ka_S
  \geq v \}$ has a unique minimal element by
  \cite[Prop.~2.1]{knutson.lam.ea:projections}, it follows that $X^v \cap Y$ is
  a single $B_S^-$-stable Schubert variety in $Y$. Similarly, $Y \cap X_u$ is
  either empty or a $B_S$-stable Schubert variety. Part (a) follows from this.
  Part (b) follows from \Proposition{flagsubvar}(c) because $u \ka^{-1} \in W_S$
  is equivalent to $\ka_S \leq u \leq w_0^Y \ka_S$. \Proposition{flagsubvar}(c)
  also implies that $Y_{u'}^{v'}$ and $X_{u' \ka_S}^{v' \ka_S}$ have the same
  $T$-fixed points, which proves part (c), using that $X_{u' \ka_S}^{v' \ka_S}$
  is $T$-convex in $X$ \cite{buch.chaput.ea:equivariant}.
\end{proof}

\begin{cor}\label{cor:flagsubint}%
  Let $X = G/P_X$ be a flag variety, and let $Y = G_S \ka.P_X$ be a flag
  subvariety of $X$.\smallskip

  \titem{a} Any flag subvariety of $Y$ is also a flag subvariety of
  $X$.\smallskip

  \titem{b} Any flag subvariety of $X$ contained in $Y$ is also a flag
  subvariety of $Y$.\smallskip

  \titem{c} Any non-empty intersection of flag subvarieties of $X$ is a flag
  subvariety.\smallskip

  \titem{d} Given $S_1, S_2 \subset \Delta$ and $\ka \in W$, we have
  $G_{S_1}\ka.P_X \cap G_{S_2}\ka.P_X = G_{S_1\cap S_2}\ka.P_X$.
\end{cor}
\begin{proof}
  Part (a) is clear from the definitions, and parts (b) and (c) follow from part
  (d). Let $S_1$, $S_2$, and $\ka$ be as in (d). Since both $G_{S_1}\ka.P_X \cap
  G_{S_2}\ka.P_X$ and $G_{S_1\cap S_2}\ka.P_X$ are Richardson varieties in $X$
  by \Proposition{flagsubvar}(b) and \Proposition{flagsubrich}(a,c), it is
  enough to show that these varieties have the same $T$-fixed points.
  Equivalently, we must show that $W_{S_1} \ka W_X \cap W_{S_2} \ka W_X =
  W_{S_1\cap S_2} \ka W_X$ holds as subsets of $W$. This is equivalent to $W_X
  \ka^{-1} W_{S_1} \cap W_X \ka^{-1} W_{S_2} = W_X \ka^{-1} W_{S_1 \cap S_2}$.
  Since $W_X \ka^{-1}$ is an interval in the Bruhat order of $W$, this follows
  from \cite[Thm.~2.6.1]{bjorner.brenti:combinatorics}.
\end{proof}

The \emph{support} of $w \in W$ is the subset $S \subset \Delta$ of simple roots
$\be$ for which $s_\be$ occurs in any reduced expression of $w$.

\begin{defn}\label{defn:minembed}%
  Let $X_u^v \subset X$ be the Richardson variety defined by $v \leq u \in W^X$.
  The \emph{minimal embedding} of $X_u^v$ is the flag subvariety
  $Y = G_S u.P_X$, where $S \subset \Delta$ is the support of $u v^{-1}$.
\end{defn}

It follows from \Proposition{flagsubrich}(b) that the minimal embedding of
$X_u^v$ is the unique minimal flag subvariety of $X$ that contains $X_u^v$. The
minimal embedding depends on the chosen Borel subgroup $B \subset G$, as shown
by the following example.

\begin{example}
  Let $X = \Fl(4)$ be the variety of complete flags in $\bK^4$. The minimal
  embedding of the Richardson variety $X_u^v$ defined by $u = s_1 s_2 s_3$ and
  $v = s_1 s_3$ is $G_{\{\be_1,\be_2\}} u.P_X \cong \Fl(3)$. The translate
  $s_2.X_u^v = X_{s_2 u}^{s_2 v} = G_{\{\be_1\}} s_2 u.P_X \cong \bP^1$ is its
  own minimal embedding.
\end{example}

\subsection{Product decompositions}

It was shown in \cite{buch.ravikumar:pieri} that certain Richardson varieties in
classical cominuscule flag varieties are products of smaller Richardson
varieties. We proceed to generalize this result. Two subsets $S_1, S_2 \subset
\Phi$ will be called \emph{orthogonal} if every root in $S_1$ is orthogonal to
every root in $S_2$. If $S_1, S_2 \subset \Delta$ consist of simple roots, then
this holds if and only if $S_1 \cap S_2 = \emptyset$ and no edge in the Dynkin
diagram of $\Phi$ connects a root in $S_1$ to a root in $S_2$.

\begin{prop}\label{prop:flagvar-decomp}%
  Let $X$ be a flag variety of $G$, and let $v \leq \ka \leq u \in W$. Assume
  that $\Delta = S_1 \coprod S_2$ is a disjoint union of orthogonal subsets. Set
  $Y_i = G_{S_i} \ka.P_X$, $P_{\ov Y_i} = G_{S_{3-i}}P_X$, and $\ov Y_i =
  G/P_{\ov Y_i}$, for $i \in \{1,2\}$.\smallskip

  \titem{a} The map $g.P_X \mapsto g.P_{\ov Y_i}$, for $g \in G_{S_i}$, is
  an isomorphism $Y_i \cong \ov Y_i$.\smallskip

  \titem{b} The map $g.P_X \mapsto (g.P_{\ov Y_1}, g.P_{\ov Y_2})$, for $g \in
  G$, is an isomorphism $X \cong \ov Y_1 \times \ov Y_2$.\smallskip

  \titem{c} The map of part {\rm(a)} restricts to an isomorphism $Y_i \cap X_u^v
  \cong (\ov Y_i)_u^v$, and the map of part {\rm(b)} restricts to an isomorphism
  $X_u^v \cong (\ov Y_1)_u^v \times (\ov Y_2)_u^v$.\smallskip

  \titem{d} The map $w \mapsto (w^{\ov Y_1}, w^{\ov Y_2})$ is an order
  isomorphism $W^X \cong W^{\ov Y_1} \times W^{\ov Y_2}$ of Bruhat orders, with
  inverse $(w_1, w_2) \mapsto w_1 w_2$, and $(W_{S_i})^{Y_i} = W^{\ov Y_i} =
  W_{S_i} \cap W^X$.\smallskip

  \titem{e} We have $\cP_{Y_i} = \cP_{\ov Y_i} = \Phi_{S_i} \cap \cP_X$ and
  $\cP_X = \cP_{Y_1} \coprod \cP_{Y_2}$.
\end{prop}
\begin{proof}
  The assumptions imply that $\Phi = \Phi_{S_1} \coprod \Phi_{S_2}$, where
  $\Phi_{S_1}$ and $\Phi_{S_2}$ are orthogonal subsets of $\Phi$. It follows
  that the set of products $P_{\ov Y_i} = G_{S_{3-i}} P_X$ is a parabolic
  subgroup of $G$, see e.g.\ \cite[(8.2.3)]{springer:linear*1}. It also follows
  that the product map $W_{S_1} \times W_{S_2} \to W$ is an isomorphism of
  groups, which restricts to an isomorphism $(W_{S_1} \cap W_X) \times (W_{S_2}
  \cap W_X) \cong W_X$. Write $\ka = \ka_1 \ka_2$ with $\ka_i \in W_{S_i}$. Then
  the $B_{S_i}$-stable point of $Y_i$ is $1.P_{Y_i} = \ka_{3-i}.P_X$ by
  \Proposition{flagsubvar}(a). For $\be \in S_i$ we have $s_\be \ka_{3-i}.P_X =
  \ka_{3-i} s_\be.P_X$, hence $s_\be.P_{Y_i} = 1.P_{Y_i}$ holds if and only if
  $s_\be.P_X = 1.P_X$. We obtain $(W_{S_i})_{Y_i} = W_{S_i} \cap W_X = W_{S_i}
  \cap W_{\ov Y_i}$. Part (d) follows from this. In particular, we obtain $w_0^X
  = w_0^{\ov Y_1} w_0^{\ov Y_2} = w_0^{Y_1} w_0^{Y_2}$, which proves part (e).
  Part (a) follows because $G_{S_i} P_{\ov Y_i} = G$ and $G_{S_i} \cap P_{\ov
  Y_i} = G_{S_i} \cap P_X = P_{Y_i}$, see \cite[(8.4.3)]{springer:linear*1}.
  Similarly, the identity $P_{\ov Y_1} \cap P_{\ov Y_2} = P_X$ shows that the
  map of part (b) is injective, and surjectivity holds because $\dim(X) =
  \dim(\ov Y_1 \times \ov Y_2)$ by part (e). Finally, part (d) implies that
  $\dim(Y_i \cap X_u^v) = \dim((\ov Y_i)_u^v)$ and $\dim(X_u^v) = \dim ((\ov
  Y_1)_u^v \times (\ov Y_2)_u^v)$, which proves part (c).
\end{proof}

\begin{cor}\label{cor:rich-factor}%
  Let $X = G/P_X$ be a flag variety, let $v \leq \ka \leq u$ in $W^X$, and let
  $S_1, S_2 \subset \Delta$ be orthogonal subsets such that $u v^{-1} \in W_{S_1
  \cup S_2}$. Set $Y_i = G_{S_i}\ka.P_X$ for $i \in \{1,2\}$. Then $X_u^v \cong
  (Y_1 \cap X_u^v) \times (Y_2 \cap X_u^v)$ as $T$-varieties.
\end{cor}
\begin{proof}
  Set $S = S_1 \cup S_2$. Since $v_S \leq \ka_S \leq u_S = v_S$, it follows from
  \Proposition{flagsubrich} that $X_u^v$ is a Richardson variety in $Y = G_S
  \ka.P_X$. The result therefore follows from \Proposition{flagvar-decomp}(c).
\end{proof}

\subsection{Indecomposable varieties}

Cominuscule Richardson varieties that satisfy the following definition play an
important role in this paper.

\begin{defn}
  An irreducible algebraic variety of positive dimension will be called
  \emph{indecomposable} if it is not isomorphic to a product of varieties of
  strictly smaller dimension.
\end{defn}

\begin{lemma}\label{lemma:indecomp}%
  Let $Z$ be an irreducible projective variety. If $\Pic(Z)=\Z$ or $A_1(Z)=\Z$,
  then $Z$ is indecomposable.
\end{lemma}
\begin{proof}
  Assume that $Z = Z_1 \times Z_2$, where $Z_1$ and $Z_2$ have positive
  dimension. Choose a point $p_i \in Z_i$, a curve $C_i \subset Z_i$, and an
  embedding $\phi_i : Z_i \subset \bP^N$, for $i \in \{1,2\}$. Then pushforward
  along $\phi_1 \times \phi_2 : Z \to \bP^N \times \bP^N$ maps $[C_1 \times
  p_2]$ and $[p_1 \times C_2]$ to linearly independent classes in $A_1(\bP^N
  \times \bP^N)$, and pullback along $\phi_1 \times \phi_2$ maps the generators
  of $\Pic(\bP^N \times \bP^N)$ to linearly independent classes in $\Pic(Z)$.
\end{proof}

\begin{defn}\label{defn:poset-connected}%
  Let $\cP$ be a partially ordered set. Two elements $\al_1, \al_2 \in \cP$ are
  \emph{comparable} if $\al_1 \leq \al_2$ or $\al_2 \leq \al_1$; otherwise
  $\al_1$ and $\al_2$ are \emph{incomparable}. Two subsets $A_1, A_2 \subset
  \cP$ are \emph{incomparable} if no element from $A_1$ is comparable to an
  element from $A_2$. A \emph{connected component} of $\cP$ is a maximal
  non-empty subset that is not a union of two non-empty incomparable subsets.
  The partially ordered set $\cP$ is \emph{connected} if it is a connected
  component of itself.
\end{defn}

Notice that \Definition{poset-connected} implies that the empty set is not
connected and has no connected components.

\begin{lemma}\label{lemma:PXdecomp}%
  Let $X = G/P_X$, and assume that $\cP_X = \cP_1 \coprod \cP_2$ is a union of
  two incomparable subsets. Let $S_i \subset \Delta$ be the subset of all simple
  roots that are connected to $\min(\cP_i)$ by a path in the Dynkin diagram of
  $\Phi$, for $i \in \{1,2\}$. Then $S_1 \cap S_2 = \emptyset$.
\end{lemma}
\begin{proof}
  If $S_1 \cap S_2 \neq \emptyset$, then $S_1 \cap S_2$ contains a connected
  component $\Delta' \subset \Delta$ of the Dynkin diagram, hence both $\cP_1$
  and $\cP_2$ have non-empty intersection with $\Delta'$. This contradicts the
  assumption that $\cP_1$ and $\cP_2$ are incomparable, as \eqn{min2PX} implies
  that the highest root of $\Span_\Z(\Delta') \cap \Phi$ belongs to either
  $\cP_1$ or $\cP_2$.
\end{proof}

\begin{cor}\label{cor:indecomp-flag}%
  Let $X = G/P_X$ be a flag variety.\smallskip

  \titem{a} If $\# \min(\cP_X) = 1$, then $X$ is indecomposable.\smallskip

  \titem{b} If $X$ is indecomposable, then $\cP_X$ is connected.
\end{cor}
\begin{proof}
  This follows from \Lemma{indecomp}, \Lemma{PXdecomp}, and
  \Proposition{flagvar-decomp}.
\end{proof}

\begin{conj}
  The flag variety $X = G/P_X$ is indecomposable if and only if $\cP_X$ is
  connected.
\end{conj}

\begin{remark}
  Let $X = G/P_X$ be a flag variety, and let $X_u^v$ be the Richardson variety
  defined by $v \leq u \in W^X$. If $X_u^v$ is indecomposable, then the support
  of $u v^{-1}$ is a connected subset of the Dynkin diagram of $\Phi$ by
  \Corollary{rich-factor}. The converse holds when $X$ is cominuscule, but is
  not true in general. For example, for $X = \Fl(4)$ we have $X_{4132}^{1423} =
  \{ V_\bull \in X \mid V_2 = \langle e_1, e_4 \rangle \} \cong \bP^1 \times
  \bP^1$.
\end{remark}

% !TeX root=comrich.tex

\section{Cominuscule Richardson varieties}\label{sec:comin}

In this section we interpret the results of \Section{decomp} for cominuscule
flag varieties. We prove that flag subvarieties of cominuscule flag varieties
are again cominuscule, and show that any cominuscule Richardson variety is a
product of indecomposable cominuscule Richardson varieties.

\subsection{Cominuscule flag varieties}\label{sec:comin-flagvar}%

Given a simple root $\be \in \Delta$, we let $\om_\be^\vee$ denote the
corresponding fundamental coweight. We have $\al = \sum_{\be \in \Delta}
(\al,\om_\be^\vee) \be$ for $\al \in \Phi$.

\begin{defn}\label{defn:comin}%
  The flag variety $X = G/P_X$ is called \emph{cominuscule} if, for each root
  $\al \in \Phi$, we have
  \[
    \sum_{\ga \in \min(\cP_X)} (\al, \om_\ga^\vee) \ \in \ \{\pm 1, 0\} \,.
  \]
\end{defn}

An indecomposable cominuscule flag variety is the same as what is called a
cominuscule flag variety in many other sources; see
\Corollary{indecomp-comin-classify} below. A representative collection of
indecomposable cominuscule flag varieties is shown in \Table{labsets}. Notice
that \Definition{comin} depends on the group $G$ acting on $X$. For example, the
maximal orthogonal Grassmannian $\OG(m,2m) \cong \OG(m-1,2m-1)$ is cominuscule
for the action of $\SO(2m)$, but not for the action of $\SO(2m-1)$.

% !TeX root=comrich.tex

\def\vmm#1{\vspace{#1mm}}
\begin{table}[ht]
\caption{Partially ordered sets $\cP_X$ of indecomposable cominuscule flag varieties. In each case the partial order is given by $\al' \leq \al$ if and only if $\al'$ is weakly northwest of $\al$. The numbers in the boxes indicate the labeling $\delta_X : \cP_X \to \Delta$ by simple roots.}
\label{tab:labsets}%
\begin{tabular}{|c|c|}
\hline%%%%%%%%%%%%%%%%%%%%%%%%%%%%%%%%%%%%%%%%%%%%%%%%%%%%%%%%%%%%%%%%%%%%%
&\vmm{-2}\\
Grassmannian $\Gr(3,7)$ of type A & Max.\ orthog.\ Grassmannian $\OG(6,12)$
\\
&\vmm{-3}\\
\pic{1}{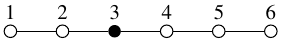} &\\
&\\
&\vmm{-3}\\
$\tableau{12}{
3 & 4 & 5 & 6 \\
2 & 3 & 4 & 5 \\
1 & 2 & 3 & 4
}$
&\vmm{-27}\\
&\pic{1}{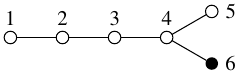}
\\ &
$\tableau{12}{
6 & 4 & 3 & 2 & 1 \\
  & 5 & 4 & 3 & 2 \\
  &   & 6 & 4 & 3 \\
  &   &   & 5 & 4 \\
  &   &   &   & 6
}$
\\ &
\vmm{-2}\\
\hline%%%%%%%%%%%%%%%%%%%%%%%%%%%%%%%%%%%%%%%%%%%%%%%%%%%%%%%%%%%%%%%%%%%%%
&\vmm{-2}\\
Lagrangian Grassmannian $\LG(6,12)$ & Cayley Plane $E_6/P_6$
\\
&\vmm{-3}\\
\pic{1}{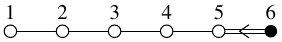} & \\
&\vmm{-2}\\
$\tableau{12}{
6 & 5 & 4 & 3 & 2 & 1 \\
  & 6 & 5 & 4 & 3 & 2 \\
  &   & 6 & 5 & 4 & 3 \\
  &   &   & 6 & 5 & 4 \\
  &   &   &   & 6 & 5 \\
  &   &   &   &   & 6
}$
&\vmm{-36}\\
& \pic{1}{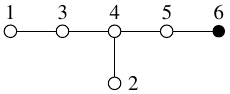} \\
&\vmm{-2}\\
&
$\tableau{12}{
6 & 5 & 4 & 2 \\
  &   & 3 & 4 & 5 & 6 \\
  &   & 1 & 3 & 4 & 5 \\
  &   &   &   & 2 & 4 & 3 & 1
}$
\\
& \vmm{-2}\\
\hline%%%%%%%%%%%%%%%%%%%%%%%%%%%%%%%%%%%%%%%%%%%%%%%%%%%%%%%%%%%%%%%%%%%%%
& \vmm{-2}\\
Even quadric $Q^{10} \subset \bP^{11}$ & Freudenthal variety $E_7/P_7$
\\
&\vmm{-3}\\
\pic{1}{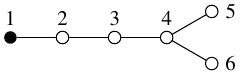} &
\\
&\vmm{-3}\\
$\tableau{12}{
1 & 2 & 3 & 4 & 5 \\
  &   &   & 6 & 4 & 3 & 2 & 1
}$
& \\
&\\
\hhline{-~}%%%%%%%%%%%%%%%%%%%%%%%%%%%%%
&\vmm{-2}\\
Odd quadric $Q^{11} \subset \bP^{12}$ &
\\
&\vmm{-3}\\
\pic{1}{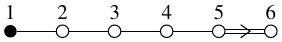} & \\
&\vmm{-1}\\
$\tableau{12}{
1 & 2 & 3 & 4 & 5 & 6 & 5 & 4 & 3 & 2 & 1
}$
& \vmm{-52}\\
& \pic{1}{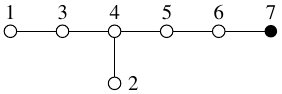} \\
&\vmm{-3}\\
&
$\tableau{12}{
7 & 6 & 5 & 4 & 3 & 1 \\
  &   &   & 2 & 4 & 3 \\
  &   &   &   & 5 & 4 & 2 \\
  &   &   &   & 6 & 5 & 4 & 3 & 1 \\
  &   &   &   & 7 & 6 & 5 & 4 & 3 \\
  &   &   &   &   &   &   & 2 & 4 \\
  &   &   &   &   &   &   &   & 5 \\
  &   &   &   &   &   &   &   & 6 \\
  &   &   &   &   &   &   &   & 7
}$
\vmm{-2}\\
& \\
\hline%%%%%%%%%%%%%%%%%%%%%%%%%%%%%%%%%%%%%%%%%%%%%%%%%%%%%%%%%%%%%%%%%%%%%
\end{tabular}
\end{table}

\begin{lemma}\label{lemma:indecomp-comin}%
  Let $X = G/P_X$ be a cominuscule flag variety. The following are equivalent:\,
  {\rm(a)} $X$ is indecomposable.\, {\rm(b)} $\cP_X$ is connected.\, {\rm(c)}
  $\# \min(\cP_X) = 1$.
\end{lemma}
\begin{proof}
  In view of \Corollary{indecomp-flag}, it suffices to show that (b) implies
  (c). If $\cP_X$ is connected, then $\min(\cP_X)$ is contained in a single
  connected component $\Delta' \subset \Delta$ of the Dynkin diagram. By letting
  $\al$ be the highest root of $\Span_\Z(\Delta') \cap \Phi$ in
  \Definition{comin}, we obtain $\# \min(\cP_X) = 1$, as required.
\end{proof}

A simple root $\ga \in \Delta$ is called \emph{cominuscule} if $(\al,
\om_\ga^\vee) \in \{\pm 1, 0\}$ for all $\al \in \Phi$. A non-trivial linear
algebraic group will be called \emph{simple} if it is semi-simple and its root
system is irreducible.

\begin{cor}\label{cor:indecomp-comin-classify}%
  The flag variety $X = G/P_X$ is an indecomposable cominuscule flag variety if
  and only if $P_X$ is a maximal parabolic subgroup of $G$ defined by a
  cominuscule simple root. In this case there exists a simple subgroup $G'
  \subset G$ such that $G'$ acts transitively on $X$, and $X$ is cominuscule for
  this action.
\end{cor}
\begin{proof}
  The first claim follows from \Lemma{indecomp-comin}. For the second claim, let
  $S_1 \subset \Delta$ be the connected component in the Dynkin diagram of the
  unique cominuscule simple root in $\min(\cP_X)$, and set $S_2 = \Delta \ssm
  S_1$. Then $X$ is a flag variety of $G_{S_1}$ by \Proposition{flagvar-decomp}.
  Now let $G' = (G_{S_1},G_{S_1})$ be the commutator subgroup.
\end{proof}

\begin{cor}\label{cor:comin-factor}%
  A flag variety is cominuscule if and only if it is a product of indecomposable
  cominuscule flag varieties.
\end{cor}
\begin{proof}
  It follows from \Proposition{flagvar-decomp} and \Lemma{indecomp-comin} that
  any cominuscule flag variety is a product of indecomposable cominuscule
  factors. On the other hand, if $X = X_1 \times \dots \times X_m$ is any
  product of cominuscule flag varieties $X_i = G_i/P_{X_i}$, then $X$ is a
  cominuscule flag variety of $G_1 \times \dots \times G_m$.
\end{proof}

\subsection{Shapes of boxes}

Let $X$ be a cominuscule flag variety and let $\cP_X \subset \Phi^+$ be the
associated partially ordered set. A lower order ideal $\la \subset \cP_X$ will
be called a \emph{straight shape} in $\cP_X$. Straight shapes generalize the
Young diagrams used for Grassmannians of type A; for this reason, the roots in
$\cP_X$ will be called \emph{boxes}. The boxes in $\cP_X$ will be designated as
\emph{long} or \emph{short} according to their lengths as roots in $\Phi$. By
convention, all roots in simply-laced components of $\Phi$ are considered long.
When $X$ is indecomposable, all boxes of $\cP_X$ are long unless $X$ is a
Lagrangian Grassmannian or an odd quadric. A box in $\cP_{\LG(n,2n)}$ is short
if and only if it is not on the southwest diagonal, while the only short box in
$\cP_{Q^{2n+1}}$ is the middle box (see \Table{labsets}). Given straight shapes
$\mu \subset \la$ in $\cP_X$, we let $\la/\mu$ denote the set of boxes in $\la$
that are not in $\mu$. A subset $\theta \subset \cP_X$ of the form $\theta =
\la/\mu$ is called a \emph{skew shape}. Equivalently, $\theta$ is a convex
subset of $\cP_X$, that is, $\al \in \theta$ holds whenever $\al' < \al < \al''$
and $\al', \al'' \in \theta$.

Parts (a) and (b) of the following result were proved by Proctor
\cite{proctor:bruhat}, and part (c) was proved by Stembridge
\cite{stembridge:fully}. Notice that the general statement follows easily from
the indecomposable case via \Proposition{flagvar-decomp}.

\begin{thm}\label{thm:comin-poset}%
  Let $X$ be a cominuscule flag variety.\smallskip

  \titem{a} The Bruhat order on $W^X$ is a distributive lattice that agrees
  with the weak left Bruhat order.\smallskip

  \titem{b} The map $u \mapsto I(u)$ is a bijection from $W^X$ to the set
  of straight shapes in $\cP_X$. We have $v \leq u$ in $W^X$ if and only if
  $I(v) \subset I(u)$.
  \smallskip

  \titem{c} All elements of $W^X$ are fully commutative.
\end{thm}

Given a straight shape $\la \subset \cP_X$, we let $w_\la \in W^X$ denote the
unique element with inversion set $I(w_\la) = \la$. The Schubert and Richardson
varieties in $X$ will frequently be indexed by straight shapes using the
notation $X_\la = X_{w_\la}$, $X^\mu = X^{w_\mu}$, and $X_\la^\mu = X_\la \cap
X^\mu$, where $\mu \subset \la \subset \cP_X$ are straight shapes. Notice that
$\dim(X_\la^\mu) = |\la/\mu|$ is the number of boxes in the skew shape
$\la/\mu$.

An ordering $\mu = \{\al_1, \al_2, \dots, \al_{|\mu|}\}$ of the boxes of the
straight shape $\mu$ is called \emph{increasing} if $\al_i < \al_j$ implies $i <
j$. In this case we have \cite[Thm.~2.4]{buch.samuel:k-theory}
\begin{equation}\label{eqn:comin-rootprod}%
  w_\mu = s_{\al_1} s_{\al_2} \cdots s_{\al_{|\mu|}} \,.
\end{equation}
For any root $\al \in\cP_X$, consider the shape $\la(\al) = \{ \al' \in \cP_X
\mid \al' < \al \}$, and set $\delta_X(\al) = w_{\la(\al)}.\al$. Since
$s_{\delta_X(\al)} = w_{\la(\al)} s_\al w_{\la(\al)}^{-1} =
w_{\la(\al)\cup\{\al\}} w_{\la(\al)}^{-1}$ has length 1 by
\Theorem{comin-poset}, it follows that $\delta_X : \cP_X \to \Delta$ is a
labeling of the boxes in $\cP_X$ by simple roots, such that
\begin{equation}\label{eqn:comin-redexp}%
  w_\mu = s_{\delta_X(\al_{|\mu|})} \cdots s_{\delta_X(\al_2)} s_{\delta_X(\al_1)}
\end{equation}
is a reduced expression of $w_\mu$; see \cite[Cor.~2.10]{buch.samuel:k-theory}.
Examples of this labeling are provided in \Table{labsets}. More generally, if
$\theta \subset \cP_X$ is any skew shape, and $\theta = \{\al'_1, \al'_2, \dots,
\al'_{|\theta|}\}$ is an increasing ordering, we define
\[
  w_\theta = s_{\delta_X(\al'_{|\theta|})} \cdots s_{\delta_X(\al'_2)} s_{\delta_X(\al'_1)} \,.
\]
This is a reduced expression of $w_\theta$, and we have $w_\theta = w_\la
w_\mu^{-1}$ whenever $\theta = \la/\mu$ is written as a difference between
straight shapes $\mu \subset \la \subset \cP_X$.

\begin{lemma}\label{lemma:comparable-boxes}%
  Let $\al_1, \al_2 \in \cP_X$ be boxes and let $\theta \subset \cP_X$ be a skew
  shape.\smallskip

  \titem{a} If $\al_2$ covers $\al_1$, then $\delta_X(\al_1)$ and
  $\delta_X(\al_2)$ are distinct and not orthogonal.\smallskip

  \titem{b} If $\al_1$ and $\al_2$ are incomparable, then $\delta_X(\al_1)$ and
  $\delta_X(\al_2)$ are orthogonal.\smallskip

  \titem{c}\, $\theta$ is connected if and only if $\dex(\theta)$ is a connected
  subset of the Dynkin diagram.
\end{lemma}
\begin{proof}
  This follows from \Theorem{comin-poset} and \eqn{comin-redexp}.
\end{proof}

Given a straight shape $\la \subset \cP_X$ and $u \in W$, we set $u \bullet \la
= I((u w_\la)^X)$, that is, $u \bullet \la \subset \cP_X$ is the unique straight
shape such that $u.(w_\la.P_X) = w_{u \bullet \la}.P_X$. The action of $W$ on
$X^T$ is expressed in terms of shapes as follows.

\begin{lemma}\label{lemma:act-fixed}%
  Let $\la \subset \cP_X$ be a straight shape. For $\be \in \Delta$ we have
  \[
    s_\be \bullet \la = \begin{cases}
      \la \cup \{\al\} & \text{if $\cP_X \ssm \la$ contains a minimal box $\al$ satisfying $\delta_X(\al) = \be$;} \\
      \la \ssm \{\al\} & \text{if $\la$ contains a maximal box $\al$ satisfying $\delta_X(\al) = \be$;} \\
      \la & \text{otherwise.}
    \end{cases}
  \]
\end{lemma}
\begin{proof}
  We may assume $\la \subset s_\be \bullet \la$. If $\la \neq s_\be \bullet
  \la$, then $s_\be \bullet \la = \la \cup \{\al\}$ for some box $\al \in
  \cP_X$, and \eqn{comin-redexp} implies that $\delta_X(\al) = \be$.
\end{proof}

Identify $\Z^2$ with a grid of boxes in the plane. The point $(r,c) \in \Z^2$
represents a box in row $r$ and column $c$, where the row number increases from
north to south, and the column number increases from east to west. The
\emph{northwest-to-southeast order} on $\Z^2$ is defined by $(r',c') \leq (r,c)$
if and only if $r' \leq r$ and $c' \leq c$. This is a distributive lattice with
meet $(r_1,c_1) \wedge (r_2,c_2) = (\min(r_1,r_2), \min(c_1,c_2))$ and join
$(r_1,c_1) \vee (r_2,c_2) = (\max(r_1,r_2), \max(c_1,c_2))$. Given simple roots
$\be_1, \be_2 \in \Delta$ in the same connected component of the Dynkin diagram
of $\Phi$, we let $[\be_1, \be_2] \subset \Delta$ denote the subset of simple
roots occurring on the shortest path from $\be_1$ to $\be_2$ in the Dynkin
diagram. The following result implies that $\cP_X$ is a distributive lattice
when $X$ is indecomposable.

\begin{lemma}\label{lemma:PX-lattice}%
  Let $X$ be an indecomposable cominuscule flag variety.\smallskip

  \titem{a}
  The poset $\cP_X$ is isomorphic to a subset of $\Z^2$ that is closed under the
  join and meet operations. Moreover, every covering $\al_1 \lessdot \al_2$ in
  $\cP_X$ is also a covering in $\Z^2$.\smallskip

  \titem{b}
  Let $\al_1, \al_2 \in \cP_X$ be incomparable boxes. For all $\al_1 \wedge
  \al_2 \leq \al \leq \al_1 \vee \al_2$ we have $\delta_X(\al) \in
  [\delta_X(\al_1), \delta_X(\al_2)] \subset \Delta$.
\end{lemma}
\begin{proof}
  This follows by inspection of \Table{labsets}.
\end{proof}

\subsection{Cominuscule flag subvarieties}

We next show that any flag subvariety of the cominuscule flag variety $X =
G/P_X$ is also cominuscule. Given a straight shape $\la \subset \cP_X$ and a
subset $S \subset \Delta$, let $\la_S \subset \la$ be the minimal straight shape
for which $\delta_X(\la/\la_S) \subset S$, and let $\la^S \supset \la$ be the
maximal straight shape for which $\delta_X(\la^S/\la) \subset S$. With the
notation of \Section{flagsubvar} we have $w_{\la_S} = (w_\la)_S$ and $w_{\la^S}
= w_0^Y (w_\la)_S$, with $Y = G_S w_\la.P_X$; see \Proposition{flagsubvar}(c).

\begin{prop}\label{prop:comin}%
  Let $X = G/P_X$ be a cominuscule flag variety and let $Y = G_S w_\la.P_X
  \subset X$ be the flag subvariety defined by $S \subset \Delta$ and $\la
  \subset \cP_X$.\smallskip

  \titem{a} $Y = X_{\la^S}^{\la_S}$ is a cominuscule flag variety of $G_S$
  with $B_S$-stable point $w_{\la_S}.P_X$.\smallskip

  \titem{b} The map $\al \mapsto w_{\la_S}^{-1}.\al$ is an order isomorphism
  $\cP_Y \cong \la^S/\la_S$.\smallskip

  \titem{c} The labeling $\delta_Y : \cP_Y \to S$ is given by $\delta_Y(\al)
  = \delta_X(w_{\la_S}^{-1}.\al)$.\smallskip

  \titem{d} Given straight shapes $\tau \subset \sigma \subset \cP_Y$, we
  have $Y_\sigma^\tau = X_{\ov{\sigma}}^{\ov{\tau}}$, where $\ov{\tau} = \la_S
  \cup w_{\la_S}^{-1}.\tau$ and $\ov{\sigma} = \la_S \cup w_{\la_S}^{-1}.\sigma$
  are straight shapes in $\cP_X$.
\end{prop}
\begin{proof}
  It follows from \Proposition{flagsubvar} that $Y = X_{\la^S}^{\la_S}$, and
  $w_{\la_S}.P_X$ is the $B_S$-stable point in $Y$. To see that $Y$ is
  cominuscule, we may assume by \Proposition{flagvar-decomp} and
  \Corollary{comin-factor} that the Dynkin diagram $\Delta$ is connected, and
  that $S$ is a connected subset of $\Delta$. We may also assume that $Y$ is not
  a point, so that $\la_S \neq \la^S$. By \Corollary{indecomp-comin-classify} we
  must show there exists $\ga \in S$ such that $\ga$ is a cominuscule simple
  root of $\Phi' = \Span_\Z(S) \cap \Phi$ and $s_\be \bullet \la_S = \la_S$ for
  $\be \in S \ssm \{\ga\}$. By construction of $\la_S$, no maximal box $\al'$ of
  $\la_S$ satisfies $\delta_X(\al') \in S$. Assume that two minimal boxes $\al_1
  \neq \al_2$ of $\cP_X \ssm \la_S$ satisfy $\delta_X(\al_1), \delta_X(\al_2)
  \in S$. Then $\al_1 \wedge \al_2 \in \la_S$. Let $\al' \in \la_S$ be maximal
  with $\al' \geq \al_1 \wedge \al_2$. Since $\al_1 \not\leq \al'$ and $\al_2
  \not\leq \al'$, it follows from \Lemma{PX-lattice}(a) that $\al' \leq \al_1
  \vee \al_2$, after which \Lemma{PX-lattice}(b) implies $\delta_X(\al') \in S$,
  a contradiction. We deduce that $\cP_X \ssm \la_S$ contains a unique minimal
  box $\al$ with $\delta_X(\al) \in S$. Set $\ga = \delta_X(\al)$. By
  \Lemma{act-fixed} we have $s_\be \bullet \la_S = \la_S$ for $\be \in S \ssm
  \{\ga\}$; we must show that $\ga$ is a cominuscule simple root of $\Phi'$. If
  $\Phi'$ has type A, then any simple root $\ga$ is cominuscule, so assume that
  $\Phi'$ is not of type A. Then $S$ contains either two simple roots of
  distinct lengths, or a simple root connected to three other simple roots in
  the Dynkin diagram of $\Phi'$. If $X$ is a Lagrangian Grassmannian or a
  maximal orthogonal Grassmannian, it follows that $\delta_X(\al') \in S$ for
  all boxes $\al'$ on the southwest diagonal of $\cP_X$. Since some box $\al'$
  on the southwest diagonal is either maximal in $\la_S$ or minimal in $\cP_X
  \ssm \la_S$, we deduce that $\al = \al'$ and $\ga$ is cominuscule. If $X$ is a
  quadric, then $\ga$ is the simple root in $S$ of lowest index using the
  numbering of \Table{labsets}, hence $\ga$ is cominuscule in $\Phi'$. Finally,
  it has been checked using a computer that $\ga$ is cominuscule in the finitely
  many cases for $E_6/P_6$ and $E_7/P_7$. This proves part (a). Part (b) is a
  special case of \cite[Lemma~4.2(a)]{buch.chaput.ea:positivity}; see also
  \Proposition{flagsubvar}(d). Given any straight shape $\sigma \subset \cP_Y$,
  it follows from part (b) that $\ov{\sigma} = \la_S \cup w_{\la_S}^{-1}.\sigma$
  is a straight shape in $\cP_X$, and \eqn{comin-rootprod} implies that
  $w_{\ov{\sigma}} = w_\sigma w_{\la_S}$. Part (d) therefore follows from
  \Proposition{flagsubrich}(c). Let $\al \in \cP_Y$ and set $\sigma = \{ \al'
  \in \cP_Y \mid \al' < \al \}$, so that $\delta_Y(\al) = w_\sigma.\al$. Since
  $w_{\la_S}^{-1}.\al$ is a minimal box of $\cP_X \ssm \ov{\sigma}$, it follows
  from \cite[Lemma~4.2(b)]{buch.chaput.ea:positivity} that
  $\delta_X(w_{\la_S}^{-1}.\al) = w_{\ov{\sigma}}.(w_{\la_S}^{-1}.\al) =
  \delta_Y(\al)$, which proves part (c).
\end{proof}

\subsection{Richardson varieties defined by skew shapes}

The next result shows that any skew shape $\theta \subset \cP_X$ defines a
Richardson variety $X_\theta$ in $X$ up to translation.

\begin{prop}\label{prop:comin-rich-trans}%
  Let $X = G/P_X$ be a cominuscule flag variety, and let $X_\la^\mu$ and
  $X_{\la'}^{\mu'}$ be Richardson varieties in $X$ such that $\la/\mu =
  \la'/\mu' \subset \cP_X$. Let $S \subset \Delta$ be the subset of simple roots
  that are orthogonal to $\delta_X(\la/\mu)$.\smallskip

  \titem{a} For $w \in W_S$ we have $w\bullet\mu \subset w\bullet\la$,
  $(w\bullet\la)/(w\bullet\mu) = \la/\mu$, and $X_{w\bullet\la}^{w\bullet\mu} =
  w.X_\la^\mu$.\smallskip

  \titem{b} There exists $w \in W_S$ such that $w.X_\la^\mu =
  X_{\la'}^{\mu'}$.\smallskip

  \titem{c} Assume that $X_\la^\mu$ and $X_{\la'}^{\mu'}$ are both contained in
  a larger Richardson variety $X_a^b$. Then $[X_\la^\mu] = [X_{\la'}^{\mu'}]$
  holds in the Chow group $A_{|\la/\mu|}(X_a^b)$.
\end{prop}
\begin{proof}
  We first prove part (a) when $w = s_\be$ is a simple generator of $W_S$, given
  by $\be \in S$, such that $s_\be\bullet\mu \neq \mu$. By \Lemma{act-fixed} we
  have either $s_\be\bullet\mu = \mu\cup\{\al\}$ or $s_\be\bullet\mu =
  \mu\ssm\{\al\}$, with $\delta_X(\al) = \be$, and the definition of $S$ gives
  $\al \notin \la/\mu$. If $s_\be\bullet\mu = \mu\cup\{\al\}$, then $\al$ is a
  minimal box of both $\cP_X\ssm\mu$ and $\cP_X\ssm\la$, hence $s_\be\bullet\la
  = \la\cup\{\al\}$. Otherwise we have $s_\be\bullet\mu = \mu\ssm\{\al\}$, where
  $\al$ is a maximal box of $\mu$. If $\al<\al'$ is a covering in $\cP_X$, then
  \Lemma{comparable-boxes}(a) and the definition of $S$ show that $\al' \notin
  \la/\mu$, hence $\al$ is also a maximal box of $\la$, and we have
  $s_\be\bullet\la = \la\ssm\{\al\}$. In either case we obtain $s_\be\bullet\mu
  \subset s_\be\bullet\la$ and $(s_\be\bullet\la)/(s_\be\bullet\mu) = \la/\mu$.
  We also have $s_\be.X_\la^\mu = X_{s_\be\bullet\la}^{s_\be\bullet\mu}$, as
  both varieties are $T$-convex in $X$ \cite{buch.chaput.ea:equivariant} and
  contain the same $T$-fixed points by \Lemma{act-fixed}. A symmetric argument
  proves part (a) when $w = s_\be$ and $s_\be\bullet\la \neq \la$, and the
  general case of (a) follows from these cases. For parts (b) and (c), we may
  assume that $\la' = \la\cup\{\al\}$ and $\mu' = \mu\cup\{\al\}$ for some box
  $\al \in \cP_X\ssm\la$. Set $\be = \delta_X(\al)$. Since $\al$ is incomparable
  to $\la/\mu$, we obtain $\be \in S$ by \Lemma{comparable-boxes}(b), so part
  (a) shows that $s_\be.X_\la^\mu = X_{\la'}^{\mu'}$. This proves part (b). The
  group $G_{\{\be\}}$ stabilizes $X_{\la'}^\mu$, as it stabilizes both
  $X_{\la'}$ and $X^\mu$. We deduce that $[X_\la^\mu] = [X_{\la'}^{\mu'}]$ holds
  in $A_*(X_{\la'}^{\mu})$. Part (c) follows from this since $X_{\la'}^\mu
  \subset X_a^b$.
\end{proof}

\begin{defn}\label{defn:skew-rich}%
  Given any skew shape $\theta \subset \cP_X$, we let $X_\theta$ denote the
  isomorphism class of the Richardson variety $X_\la^\mu \subset X$, for any
  choice of straight shapes $\mu \subset \la$ in $\cP_X$ such that $\la/\mu =
  \theta$. Moreover, if $\phi \subset \theta$ is any smaller skew shape, then we
  set $[X_\phi] = [X_{\la'}^{\mu'}] \in A_{|\phi|}(X_\theta)$, for any choice of
  straight shapes $\mu', \la' \subset \cP_X$ such that $\mu \subset \mu' \subset
  \la' \subset \la$ and $\phi = \la'/\mu'$.
\end{defn}

It follows from \Proposition{comin-rich-trans} that the variety $X_\theta$ and
class $[X_\phi] \in A_*(X_\theta)$ are well-defined. One way to choose the
shapes $\mu', \la' \subset \cP_X$ in \Definition{skew-rich} is to let $\la'$ be
the order ideal generated by $\mu \cup \phi$ and set $\mu' = \la' \ssm \phi$.

\begin{remark}
  More generally, if $\phi \subset \theta \subset \cP_X$ are skew shapes, then
  there exists a well-defined group homomorphism $A(X_\phi) \to A(X_\theta)$
  that agrees with the map $A(X_{\la'}^{\mu'}) \to A(X_\la^\mu)$ whenever $\mu
  \subset \mu' \subset \la' \subset \la$ are straight shapes such that $\phi =
  \la'/\mu'$ and $\theta = \la/\mu$. We will not need this fact.
\end{remark}

\begin{remark}\label{remark:comin-rich-int}%
  Let $X_{\la'}^{\mu'}$ and $X_{\la''}^{\mu''}$ be Richardson varieties in $X$,
  and set $\la = \la' \cap \la''$ and $\mu = \mu' \cup \mu''$. Using that
  $X_{\la'} \cap X_{\la''} = X_\la$ and $X^{\mu'} \cap X^{\mu''} = X^\mu$, we
  obtain
  \[
    X_{\la'}^{\mu'} \cap X_{\la''}^{\mu''} = X_\la \cap X^\mu = \begin{cases}
      X_\la^\mu & \text{if $\mu \subset \la$;} \\
      \emptyset & \text{if $\mu \not\subset \la$.}
    \end{cases}
  \]
  When this intersection is not empty, the isomorphism class of $X_\la^\mu$ is
  given by the skew shape $\la/\mu = (\la'/\mu') \cap (\la''/\mu'')$.
\end{remark}

\begin{lemma}\label{lemma:skew-embed}%
  Let $\theta \subset \cP_X$ be a skew shape and set $S = \delta_X(\theta)$.
  There exists a unique maximal skew shape $\wh\theta \subset \cP_X$
  containing $\theta$ such that $\delta_X(\wh\theta) = S$. Moreover, if
  $\mu \subset \la$ are straight shapes in $\cP_X$ with $\theta = \la/\mu$,
  then $\wh\theta = \la^S/\la_S$.
\end{lemma}
\begin{proof}
  Let $\mu \subset \la$ be straight shapes in $\cP_X$ such that $\theta =
  \la/\mu$, and let $\xi \subset \cP_X$ be a skew shape such that $\theta
  \subset \xi$ and $\delta_X(\xi) = S$. We must show that $\xi \subset
  \la^S/\la_S$. Choose straight shapes $\tau' \subset \mu' \subset \la' \subset
  \sigma'$ such that $\xi = \sigma'/\tau'$ and $\theta = \la'/\mu'$. Let $S'
  \subset \Delta$ be the simple roots orthogonal to $S$. By
  \Proposition{comin-rich-trans} we can choose $w \in W_{S'}$ such that $w
  \bullet \mu' = \mu$ and $w \bullet \la' = \la$. Set $\tau = w \bullet \tau'$
  and $\sigma = w \bullet \sigma'$. We obtain $\tau \subset \mu \subset \la
  \subset \sigma$ and $\xi = \sigma/\tau$. In particular, we have
  $\delta_X(\la/\tau) \subset S$ and $\delta_X(\sigma/\la) \subset S$. This
  implies $\la_S \subset \tau$ and $\sigma \subset \la^S$, and therefore $\xi
  \subset \la^S/\la_S$.
\end{proof}

The skew shape $\wh\theta$ of \Lemma{skew-embed} will be called the
\emph{minimal embedding} of $\theta$, and $X_{\wh\theta}$ will be called the
\emph{minimal embedding} of $X_\theta$. Indeed, $X_{\wh\theta}$ is isomorphic to
the minimal embedding of $X_\la^\mu$ whenever $\la/\mu = \theta$. Notice that
$X_{\wh\theta}$ is a cominuscule flag variety by \Proposition{comin}.

\begin{cor}\label{cor:comin-rich-factor}%
  Let $\theta \subset \cP_X$ be a skew shape. Assume that $\theta = \theta_1
  \cup \theta_2$, where $\theta_1$ and $\theta_2$ are incomparable subsets of
  $\cP_X$. Then $X_\theta \cong X_{\theta_1} \times X_{\theta_2}$.
\end{cor}
\begin{proof}
  Choose straight shapes $\mu \subset \la \subset \cP_X$ such that $\theta =
  \la/\mu$, and set $S_i = \delta_X(\theta_i)$ and $Y_i = G_{S_i} w_\la.P_X$ for
  $i \in \{1,2\}$. Since $w_\la w_\mu^{-1} = w_{\la/\mu} \in W_{S_1 \cup S_2}$,
  and $S_1$ and $S_2$ are orthogonal subsets of $\Delta$ by
  \Lemma{comparable-boxes}(b), it follows from \Corollary{rich-factor} that
  $X_\theta \cong (Y_1 \cap X_\la^\mu) \times (Y_2 \cap X_\la^\mu)$.
  \Proposition{comin}(a) gives $Y_i = X_{\la^{S_i}}^{\la_{S_i}}$, so it suffices
  to show that $\theta_i = \theta \cap (\la^{S_i}/\la_{S_i})$ by
  \Remark{comin-rich-int}. Using that $\theta_1$ and $\theta_2$ are
  incomparable, it follows that $\la \ssm \theta_i$ is a straight shape in
  $\cP_X$, so we obtain $\wh\theta_i = \la^{S_i}/\la_{S_i}$ by
  \Lemma{skew-embed}. The result follows from this since $\wh\theta_i \cap
  \theta_{3-i} = \emptyset$.
\end{proof}

\Corollary{comin-rich-factor} will be generalized in \Theorem{chow1rich}, where
we show that any cominuscule Richardson variety has a unique factorization as a
product of indecomposable cominuscule Richardson varieties.

\begin{remark}\label{remark:min-embed}%
  Let $X = G/P_X$ be a cominuscule flag variety and let $\theta \subset \cP_X$
  be a skew shape. Then \Lemma{skew-embed} shows that $X$ is the minimal
  embedding of $X_\theta$ if and only if $\dex(\theta) = \dex(\cP_X)$. When
  studying the Richardson variety $X_\theta$, it is frequently convenient to
  replace $X$ with the minimal embedding, so that the identity $\dex(\theta) =
  \dex(\cP_X)$ can be assumed. By \Corollary{indecomp-comin-classify} and
  \Corollary{comin-factor} we may additionally replace $G$ with a smaller group,
  so that $\dex(\theta) = \Delta$. In this case it follows from
  \Lemma{comparable-boxes}(c) that $\theta$, $\cP_X$, and the Dynkin diagram of
  $\Phi$ have the same number of connected components.
\end{remark}

% !TeX root=comrich.tex

\section{The first Chow group and Picard group}\label{sec:chowpic}

In this section we compute the first Chow group $A_1(X_\theta)$ and the Picard
group $\Pic(X_\theta)$ of any cominuscule Richardson variety $X_\theta$.

\subsection{$T$-stable curves in flag varieties}

We start by recalling some facts about $T$-stable curves in an arbitrary flag
variety $X = G/P_X$. The proofs given in \cite[\S3]{fulton.woodward:quantum} are
valid for flag varieties defined over any algebraically closed field. A
$T$-stable curve in $X$ is assumed to be irreducible. Given a root $\al \in
\cP_X$, there exists a unique $T$-stable curve $C_\al \subset X$ that contains
the points $1.P_X$ and $s_\al.P_X$. Any $T$-stable curve in $X$ has the form
$w.C_\al$ for some $\al \in \cP_X$ and $w \in W$.

Any $W_X$-invariant character $\la : T \to \bG_m$ extends to a character of
$P_X$. Let $\bK_\la$ be the 1-dimensional representation of $P_X$ defined by
$p.z = \la(p)z$, and let $L_\la = G \times^{P_X} \bK_{-\la}$ be the associated
line bundle on $X$. The total space of this bundle consists of all pairs $[g,z]$
with $g \in G$ and $z \in \bK$, subject to the relation $[g p,\la(p)z] = [g,z]$
for all $p \in P_X$. The degree of $L_\la$ restricted to $C_\al$ is given by
\begin{equation}\label{eqn:curve-degree}%
  \int_{C_\al} c_1(L_\la) \,=\, (\al^\vee, \la) \,,
\end{equation}
where $\al^\vee$ is the coroot of $\al$.

We will say that two distinct $T$-fixed points $x, y \in X^T$ are
\emph{adjacent} if both points are contained in some $T$-stable curve in $X$.
Equivalently, we have $x \neq y = s_\al.x$ for some positive root $\al \in
\Phi^+$. This root $\al \in \Phi^+$ is unique by
\cite[Lemma~4.1(i)]{fulton.woodward:quantum}. For convenience we will also say
that two elements $u, v \in W^X$ are adjacent if the corresponding $T$-fixed
points $x=u.P_X$ and $y=v.P_X$ are adjacent. In this case we denote the unique
$T$-stable curve through these points by either $C_{x,y}$ or $C_{u,v}$. Recall
from \Definition{skew-rich} that, if $X$ is a cominuscule flag variety, then any
skew shape $\phi \subset \cP_X$ defines a Chow class $[X_\phi] \in
A_{|\phi|}(X)$.

\begin{lemma}\label{lemma:comin-degree}%
  Let $X$ be an indecomposable cominuscule flag variety, and let $\ga \in
  \Delta$ be the cominuscule simple root defining $X$. For $\al \in \cP_X$ we
  have in $A_1(X)$ that
  \[
    [X_{\{\al\}}] \,=\, [C_\al] \,=\, \begin{cases}
      [X_{s_\ga}] & \text{if $\al$ is long,} \\
      2[X_{s_\ga}] & \text{if $\al$ is short.}
    \end{cases}
  \]
\end{lemma}
\begin{proof}
  It follows from equation \eqn{comin-rootprod} that the class $[X_{\{\al\}}]
  \in A_1(X)$ is represented by the curve $X_{\la(\al)\cup\{\al\}}^{\la(\al)} =
  w_{\la(\al)}.C_\al$, so it is enough to prove the second equality.
  Using equation \eqn{curve-degree} and the assumption that $\ga$ is
  cominuscule, we obtain
  \[
    \int_{C_\al} c_1(L_{\om_\ga}) \,=\,
    (\al^\vee,\om_\ga) \,=\,
    \frac{(\ga,\ga)}{(\al,\al)} (\al,\om_\ga^\vee) \,=\,
    \frac{(\ga,\ga)}{(\al,\al)} \,.
  \]
  The result follows from this since the root system $\Phi$ does not have type
  $G_2$, $\ga$ is long, and $A_1(X) \cong \Z$ is generated by the Schubert curve
  class $[C_\ga] = [X_{s_\ga}]$ by \Lemma{indecomp-comin}.
\end{proof}

Given a cominuscule flag variety $X$, we will say that two boxes $\al', \al \in
\cP_X$ are \emph{neighbors} if one of the boxes covers the other.

\begin{lemma}\label{lemma:neighbor-class}%
  Let $X$ be a cominuscule flag variety and let $\mu \subset \la \subset \cP_X$
  be straight shapes. Let $\al', \al \in \la/\mu$ be neighbor boxes such that
  $\al$ is at least as long as $\al'$. Then we have in $A_1(X_\la^\mu)$ that
  \begin{equation}\label{eqn:neighbor-class}%
    [X_{\{\al'\}}] = \begin{cases}
      [X_{\{\al\}}] & \text{if $\al'$ and $\al$ have the same length,} \\
      2[X_{\{\al\}}] & \text{if $\al'$ is shorter than $\al$.}
    \end{cases}
  \end{equation}
\end{lemma}
\begin{proof}
  Choose straight shapes $\tau, \sigma \subset \cP_X$ with $\sigma/\tau =
  \{\al', \al\}$ and $\mu \subset \tau \subset \sigma \subset \la$, and set $S =
  \dex(\sigma/\tau)$ and $Y = G_S w_\sigma.P_X$. Since $G_S$ has type $A_2$ or
  $B_2$, and $Y$ is a cominuscule flag variety of $G_S$ by \Proposition{comin},
  we must have $Y = \bP^2$ or $Y = Q^3$. Equation \eqn{neighbor-class} therefore
  holds in $A_1(Y)$ by \Lemma{comin-degree}. Since $X_\sigma^\tau$ is a Schubert
  variety in $Y$ by \Proposition{flagsubrich}, we obtain $A_1(X_\sigma^\tau)
  \cong A_1(Y)$ by \cite{mathieu:formules*1}, which shows that
  \eqn{neighbor-class} holds in $A_1(X_\sigma^\tau)$. Finally, since
  $X_\sigma^\tau \subset X_\la^\mu$, we deduce that \eqn{neighbor-class} holds
  in $A_1(X_\la^\mu)$.
\end{proof}

\begin{example}\label{example:lg24-curves}%
  Consider the quadric threefold $X = \LG(2,4) \cong Q^3$ of type $C_2$, and let
  $\ga \in \Delta$ be the cominuscule simple root defining $X$. Since $\cP_X$
  contains two long roots and one short root, it follows that each $T$-fixed
  point in $X$ is contained in three $T$-stable curves, two of which have class
  $[X_{s_\ga}]$ and generate $A_1(X)$, while the third curve has class
  $2[X_{s_\ga}]$. Since $X$ contains four $T$-fixed points, it follows that all
  distinct $T$-fixed points in $X$ are adjacent. We have $W^X = \{1, s_\ga, \ka,
  w_0^X\}$, where $\ell(\ka)=2$. Using that $C_{s_\ga,\ka}$ is the Richardson
  variety defined by the short root in $\cP_X$, we deduce from
  \Lemma{comin-degree} that the $T$-stable curves in $X$ have the degrees in the
  following picture, where single lines indicate the class $[X_{s_\ga}]$ and
  double lines indicate twice this class.
  \begin{center}
  \begin{tikzpicture}[x=30pt,y=30pt]
    \node (A) at (0,0) {$1.P_X$};
    \node (B) at (2,0) {$s_\ga.P_X$};
    \node (C) at (0,2) {$\ka.P_X$};
    \node (D) at (2,2) {$w_0^X.P_X$};
    \path (A) edge (B);
    \path (A) edge (C);
    \path (A) edge[double distance=1.5pt] (D);
    \path (B) edge[double distance=1.5pt] (C);
    \path (B) edge (D);
    \path (C) edge (D);
  \end{tikzpicture}
  \end{center}
\end{example}

\subsection{Triangles of $T$-stable curves in Richardson varieties}

We next prove some facts about triples of adjacent $T$-fixed points in arbitrary
flag varieties of Lie type different from $G_2$. Recall that a subgroup of a
Weyl group is \emph{parabolic} if it is generated by simple reflections.

\begin{lemma}\label{lemma:adjacent_rank2}%
  Let $\Phi$ be a root system of rank 3 that does not contain a subsystem of
  type $G_2$. Let $W$ be the Weyl group of $\Phi$, let $P \subset W$ be a
  parabolic subgroup, and let $\al_1, \al_2, \al_3 \in \Phi^+$ be distinct
  positive roots such that $s_{\al_1} \notin P$, $s_{\al_3} \notin P$, and
  $s_{\al_1} s_{\al_2} s_{\al_3} \in P$. Then the roots $\al_1, \al_2, \al_3$
  are linearly dependent.
\end{lemma}
\begin{proof}
  This has been checked using a computer for the root systems $A_3$, $B_3$,
  $A_2\times A_1$, $B_2\times A_1$, and $A_1\times A_1\times A_1$.
\end{proof}

\begin{prop}\label{prop:triangle}%
  Let $X = G/P_X$ be a flag variety, and assume that the root system of $G$ does
  not contain a subsystem of type $G_2$. Let $x, y, z \in X^T$ be pairwise
  adjacent $T$-fixed points. There exists a unique irreducible closed $T$-stable
  subvariety $\Omega \subset X$ such that $\Omega^T = \{x,y,z\}$. Furthermore,
  there exists $w \in W$ such that $w.\Omega$ is a Richardson variety in $X$
  whose minimal embedding is $\bP^2$, $\SG(1,4)$, or $Q^3$. In particular,
  $\Omega$ is isomorphic to either $\bP^2$ or the Schubert divisor in $Q^3$.
\end{prop}
\begin{proof}
  It suffices to show that $\Omega$ exists with the stated properties, as
  uniqueness then follows because $\Omega$ is $T$-convex. Choose $\al_1, \al_2,
  \al_3 \in \Phi^+$ such that $s_{\al_1}.x = y$, $s_{\al_2}.y = z$, and
  $s_{\al_3}.z = x$, and set $\Phi' = \Span_\R\{\al_1,\al_2,\al_3\} \cap \Phi$.
  Since any basis of simple roots in $\Phi'$ can be extended to a basis of
  $\Phi$, there exists $w \in W$ such that $w.\Phi' = \Phi_S$ for some subset $S
  \subset \Delta$. By replacing $x$, $y$, $z$ with $w.x$, $w.y$, $w.z$, we may
  assume $\Phi' = \Phi_S$. By replacing $X$ with the flag subvariety $G_S.x$, we
  may further assume that $\Phi$ has rank at most 3. By translating $x,y,z$ by a
  second element from $W$, we may assume $x = 1.P_X$. Then \Lemma{adjacent_rank2}
  implies that $\Phi$ has rank 2. Since our assumptions imply that $P_X \neq B$
  and $\Phi$ does not have type $A_1\times A_1$ or $G_2$, we deduce that $X$ is
  one of the spaces $\bP^2$, $\SG(1,4) \cong \bP^3$, or $\LG(2,4) \cong Q^3$.
  After a final translation of $x,y,z$ by an element from $W$, we obtain
  $\{x,y,z\} = \Omega^T$, where $\Omega \subset X$ is a 2-dimensional Schubert
  variety in $X$. This completes the proof.
\end{proof}

\begin{cor}\label{cor:triangle}%
  Let $X = G/P_X$ be a flag variety, and assume that the root system of $G$ does
  not contain a subsystem of type $G_2$. Let $X_u^v \subset X$ be a Richardson
  variety, and let $x, y, z \in (X_u^v)^T$ be pairwise adjacent $T$-fixed
  points. Then the curve classes $[C_{x,y}]$, $[C_{y,z}]$, $[C_{z,x}]$ generate
  a free Abelian subgroup of rank 1 in $A_1(X_u^v)$, and at least two of the
  three curve classes individually generate this subgroup.
\end{cor}
\begin{proof}
  Let $\Omega \subset X$ be as in \Proposition{triangle}. We have $A_1(\Omega) =
  \Z$. Since $\Omega^T \subset X_u^v$ and $X_u^v$ is $T$-convex
  \cite{buch.chaput.ea:equivariant}, we have $\Omega \subset X_u^v$. Since
  $\Omega$ is itself $T$-convex, the three curves $C_{x,y}$, $C_{y,z}$,
  $C_{z,x}$ are contained in $\Omega$. If $\Omega \cong \bP^2$, then $[C_{x,y}]
  = [C_{y,z}] = [C_{z,x}]$ holds in $A_1(\Omega)$; notice that $\SG(1,4)$ is
  minuscule, so all $T$-stable curves in $\SG(1,4)$ have the same degree. If
  $\Omega$ is isomorphic to the Schubert divisor in $Q^3$, then
  \Example{lg24-curves} shows that $[C_{x,y}] = 2[C_{y,z}] = 2[C_{z,x}]$ holds,
  up to permuting the three points. The result follows from this since
  $A_1(\Omega)$ is a subgroup of $A_1(X_u^v)$.
\end{proof}

\Corollary{triangle} is particularly useful for cominuscule Richardson
varieties, as these contain many triangles of $T$-stable curves by the
following result.

\begin{prop}\label{prop:comin-triangle}%
  Let $X$ be a cominuscule flag variety, and let $v, w \in W^X$ be adjacent
  elements such that $\ell(w) - \ell(v) \geq 2$. Then there exists $u \in W^X$
  such that $v < u < w$ and $u$ is adjacent to both $v$ and $w$.
\end{prop}
\begin{proof}
  We may assume that $X$ is indecomposable. The claim has been verified using a
  computer when $X$ is one of the exceptional spaces $E_6/P_6$ and $E_7/P_7$, so
  we may assume that $X = F_\om(m;N)$ is a Grassmannian of classical Lie type.
  In particular, the Weyl group $W$ is identified with a subgroup of the
  symmetric group $S_N$, see \Section{classical}. Let $\wt{W^X}$ be the set of
  all $m$-element subsets $J \subset [1,N]$ for which $i \in J$ implies $\ov{i}
  = N+1-i \notin J$, and endow this set with the partial order defined by
  \[
    \{i_1 < \dots < i_m\} \leq \{j_1 < \dots < j_m\}
    \ \ \Longleftrightarrow \ \
    i_p \leq j_p \text{ for all $p \in [1,m]$.}
  \]
  Then the map $\iota : W \to \wt{W^X}$ defined by $\iota(u) =
  \{u(1),\dots,u(m)\}$ is order-preserving, satisfies $\iota(u' u) =
  u'.\iota(u)$ for $u',u \in W$, and $\iota(u)=\iota(u')$ holds if and only if
  $u.P_X = u'.P_X$. We will say that $I, J \in \wt{W^X}$ are adjacent if
  $s_\al.I = J$ holds for some root $\al \in \Phi$. Then $u.P_X$ and $u'.P_X$
  are adjacent $T$-fixed points in $X$ if and only if $\iota(u)$ and $\iota(u')$
  are adjacent in $\wt{W^X}$.

  Assume that $X = Q^{N-2} = \OG(1,N)$ is a quadric hypersurface. Then we have
  $\wt{W^X} = \{ \{j\} \mid 1 \leq j \leq N \text{ and } 2j \neq N+1 \}$. Since
  $\tau_{i,k}(\{i\}) = \{k\}$, it follows that $\{i\}$ and $\{k\}$ are adjacent
  if and only if $\tau_{i,k} \in W$. If $N$ is odd, then all distinct points in
  $X^T$ are adjacent and the result is clear, so assume that $N=2n$ is even.
  Then $\{i\}$ and $\{k\}$ are adjacent if and only if $i\neq k$ and $i+k \neq
  2n+1$. Equivalently, $u, u' \in W^X$ are adjacent if and only if $u \neq u'$
  and $\ell(u)+\ell(u') \neq 2n-2$. Notice also that $u'<u$ holds in $W^X$ if
  and only if $\ell(u')<\ell(u)$. If $\ell(v) < n-1 < \ell(w)$, then any element
  $u \in W^X$ with $\ell(u)=n-1$ will satisfy the proposition, and otherwise we
  can choose any $u \in W^X$ such that $\ell(v) < \ell(u) < \ell(w)$.

  Assume that $X = \LG(n,2n)$ is a Lagrangian Grassmannian. Then $\wt{W^X}$ is
  the set of all $n$-element subsets $J \subset [1,2n]$ such that $j \in J$ or
  $\ov{j} \in J$ for all $j \in [1,n]$. Set $I = \iota(v)$ and $K = \iota(w)$.
  It suffices to show that there exists $J \in \wt{W^X}$ such that $I < J < K$
  and $J$ is adjacent to both $I$ and $K$. Indeed, if $J = s_\al.I$, then $u =
  (s_\al v)^X \in W^X$ satisfies the proposition. Choose $i,k \in [1,2n]$ such
  that $i < k \leq \ov{i}$ and $\tau_{i,k}(I) = K$. We have $I = (I\cap
  K)\cup\{i,\ov{k}\}$ and $K = (I\cap K)\cup\{k,\ov{i}\}$ since $I<K$, and $k-i
  \geq 2$ holds because $K$ does not cover $I$. Set $j = i+1$ and $A = (I\cap
  K)\ssm\{j,\ov{j}\}$. If $I\cap K\cap\{j,\ov{j}\} = \emptyset$, then we must
  have $j \in \{i,\ov{k},k,\ov{i}\}$, but only $j=\ov{k}$ is possible. We can
  therefore use one of the following three cases to define the required element
  $J \in \wt{W^X}$:\medskip

  \renewcommand{\myspca}{\vspace{-2.5ex}&&&\\}
  \renewcommand{\myspcb}{\vspace{-2.5ex}&&&\\}
  \renewcommand{\myspcc}{\vspace{-2.5ex}\\&&&\\}
  \begin{tabular}{|c|c|c|c|}
  \hline\myspca
  Case & $I$ & $J$ & $K$ \\
  \myspcb\hline\myspca
  $j \in I\cap K$ &
  $A\cup\{i,j,\ov{k}\}$ &
  $\tau_{j,k}(I) = A\cup\{i,k,\ov{j}\} = \tau_{i,j}(K)$ &
  $A\cup\{j,k,\ov{i}\}$ \\
  \myspcb\hline\myspca
  $\ov{j} \in I\cap K$ &
  $A\cup\{i,\ov{k},\ov{j}\}$ &
  $\tau_{i,j}(I) = A\cup\{j,\ov{k},\ov{i}\} = \tau_{j,k}(K)$ &
  $A\cup\{k,\ov{j},\ov{i}\}$ \\
  \myspcb\hline\myspca
  $j=\ov{k}$ &
  $A\cup\{i,\ov{k}\}$ &
  $\tau_{i,\ov{i}}(I) = A\cup\{\ov{k},\ov{i}\} = \tau_{\ov{k},k}(K)$ &
  $A\cup\{k,\ov{i}\}$
  \myspcc\hline
  \end{tabular}
  \medskip

  If $X = \Gr(m,n)$ is a Grassmannian of type $A_{n-1}$, then $X$ is a flag
  subvariety of $\LG(n,2n)$, so the result for $X$ follows from the case just
  established. Finally, assume that $X$ is a maximal orthogonal Grassmannian.
  Then $X$ is isomorphic to the space $\OG(n,2n+1)$ of type $B_n$. Since the
  Weyl groups of types $B_n$ and $C_n$ are isomorphic as Coxeter groups, and the
  isomorphism identifies the parabolic subgroups corresponding to $\OG(n,2n+1)$
  and $\LG(n,2n)$, the proposition for $\OG(n,2n+1)$ also follows from the case
  of $\LG(n,2n)$. This completes the proof.
\end{proof}

\begin{example}
  Let $Y = G_{\{\be_1,\be_2\}} s_3 s_2.P_X$ be the Richardson variety in $X =
  \SG(2,6)$ from \Example{fsvSG26}. We have $Y \cong \Fl(3)$, so $Y$ contains no
  triangles of $T$-stable curves, and \Corollary{triangle} provides no
  information about $A_1(Y)$.
\end{example}

\subsection{The first Chow group of a product}

We next compute the first Chow group of a product of varieties under suitable
conditions. Let $\Gamma$ be a connected solvable linear algebraic group, and let
$X$ be a projective $\Gamma$-variety. A result of Hirschowitz
\cite{hirschowitz:groupe} states that each equivariant Chow group
$A^\Gamma_d(X)$, defined as the group of $\Gamma$-stable cycles modulo rational
equivalence given by semi-invariant rational functions, is isomorphic to the
usual Chow group $A_d(X)$. This result was generalized to algebraic schemes in
\cite{fulton.macpherson.ea:intersection}. We need the following consequence.

\begin{prop}\label{prop:chow1gen}%
  Let $\Gamma$ be a connected solvable linear algebraic group, and let $X$ be a
  projective $\Gamma$-variety. Then $A_d(X)$ is generated by $\Gamma$-stable
  cycles for all $d$.
\end{prop}

Given projective varieties $X$ and $Y$, the pushforward along the projections
from $X \times Y$ defines a natural group homomorphism
\begin{equation}\label{eqn:chow1prod}%
  \rho : A_1(X \times Y) \to A_1(X) \oplus A_1(Y) \,.
\end{equation}

\begin{lemma}\label{lemma:chow1prod}%
  Let $X$ and $Y$ be projective varieties satisfying $A_0(X) \cong \Z \cong
  A_0(Y)$. Assume that a connected solvable linear algebraic group acts
  algebraically on $X$ with finitely many fixed points. Then the map
  \eqn{chow1prod} is an isomorphism.
\end{lemma}
\begin{proof}
  Let $p \in A_0(X)$ and $q \in A_0(Y)$ denote the point classes of $X$ and $Y$,
  and let $\phi : A_1(X) \oplus A_1(Y) \to A_1(X \times Y)$ be the group
  homomorphism defined by $\phi(\sigma \oplus \tau) = \sigma \times q + p \times
  \tau$, see \cite[\S 1.10]{fulton:intersection}. Since $\rho \circ \phi$ is the
  identity, it suffices to show that $\phi$ is surjective. Let $\Gamma$ be the
  solvable group acting on $X$, and extend the action to $X \times Y$ by
  $g.(x,y) = (g.x, y)$. Let $C \subset X \times Y$ be a 1-dimensional closed
  irreducible subvariety that is stable under this action. By
  \Proposition{chow1gen} it is enough to show that the class $[C] \in
  A_1(X\times Y)$ belongs to the image of $\phi$. Let $C' = \pi_X(C) \subset X$
  be the image of $C$ under the projection $\pi_X : X\times Y \to X$. Then $C'$
  is a $\Gamma$-stable irreducible closed subvariety of $X$. If $\dim(C')=1$,
  then choose $x_0 \in C'$ such that $x_0$ is not a $\Gamma$-fixed point in $X$,
  and choose $y_0 \in Y$ such that $(x_0,y_0) \in C$. Using that $C$ contains
  the orbit $\Gamma.(x_0,y_0)$, we obtain $C = C' \times \{y_0\}$ and $[C] =
  \phi([C'] \oplus 0)$. Otherwise $C' = \{x_0\}$ is a single point in $X$. In
  this case we obtain $C = \{x_0\} \times C''$, where $C'' = \pi_Y(C)$ is the
  projection of $C$ to $Y$, and therefore $[C] = \phi(0 \oplus [C''])$, as
  required.
\end{proof}

We will say that a projective variety $X$ is \emph{strongly indecomposable} if
$A_0(X) \cong \Z$, $A_1(X) \cong \Z$, and some connected solvable linear
algebraic group acts on $X$ with finitely many fixed points. This implies that
$X$ is indecomposable by \Lemma{indecomp}.

\begin{cor}\label{cor:unique-factor}%
  Let $X_1, \dots, X_r$ and $Y_1, \dots, Y_s$ be strongly indecomposable
  projective varieties, and assume that $\phi : X_1 \times \cdots \times X_r \to
  Y_1 \times \dots \times Y_s$ is an isomorphism of algebraic varieties. Then
  $r=s$, and $\phi$ is given by $\phi(x_1,\dots,x_r) = (\phi_1(x_{\sigma(1)}),
  \dots, \phi_r(x_{\sigma(r)}))$ for some permutation $\sigma \in S_r$ and
  isomorphisms $\phi_j : X_{\sigma(j)} \to Y_j$ for $1 \leq j \leq r$.
\end{cor}
\begin{proof}
  Identify $A_1(X_i) = \Z$ and $A_1(Y_j) = \Z$ such that curve classes in $X_i$
  and $Y_j$ correspond to positive integers, and set $X = X_1 \times \dots
  \times X_r$ and $Y = Y_1 \times \dots \times Y_s$. By \Lemma{chow1prod} we
  have $A_1(X) = \Z^r$ and $A_1(Y) = \Z^s$. Since $\phi_* : A_1(X) \to A_1(Y)$
  is an isomorphism, we obtain $r=s$. In addition, both $\phi_*$ and its inverse
  $(\phi^{-1})_*$ are given by matrices of non-negative integers. This implies
  that the matrix of $\phi_*$ is a permutation matrix, hence each composition
  $\pi_j \phi : X \to Y_j$ of $\phi$ with the projection to $Y_j$ depends only
  on $X_{\sigma(j)}$ for some permutation $\sigma \in S_r$. This shows that
  $\phi$ has the claimed form.
\end{proof}

Let $X_u^v$ be a Richardson variety in the flag variety $X = G/P_X$. Since the
maximal torus $T$ acts on $X_u^v$ with finitely many fixed points, the following
observation shows that \Lemma{chow1prod} applies to products of Richardson
varieties.

\begin{lemma}\label{lemma:chowzero}%
  For any Richardson variety $X_u^v$ we have $A_0(X_u^v) = \Z$.
\end{lemma}
\begin{proof}
  By \Proposition{chow1gen} we must show that all $T$-fixed points in $X_u^v$
  are rationally equivalent. Any such point has the form $w.P_X$, where $w \in
  W^X$ and $v \leq w \leq u$. There exists a chain in the Bruhat order on $W^X$,
  $v = v_0 < v_1 < \dots < v_r = w$, such that $\ell(v_i) = \ell(v)+i$ for $1
  \leq i \leq r$. Using that each Richardson variety $X_{v_i}^{v_{i-1}}$ is
  isomorphic to $\bP^1$ and contained in $X_u^v$, we obtain $[v_{i-1}.P_X] =
  [v_i.P_X] \in A_0(X_u^v)$ for all $i \in [1,r]$. This implies $[w.P_X] =
  [v.P_X]$ in $A_0(X_u^v)$, as required.
\end{proof}

\subsection{The first Chow group of a cominuscule Richardson variety}

We can now compute the first Chow group of any cominuscule Richardson variety.
Given a cominuscule flag variety $X$ and two straight shapes $\tau \subset
\sigma \subset \cP_X$, we will say that $\sigma$ and $\tau$ are \emph{adjacent}
if the corresponding points $w_\tau.P_X, w_\sigma.P_X \in X^T$ are adjacent.
Equivalently, we have $\tau \neq \sigma = s_\al \bullet \tau$ for some $\al \in
\Phi^+$. In this case we let $C_{\tau,\sigma}$ denote the unique $T$-stable
curve through $w_\tau.P_X$ and $w_\sigma.P_X$. Since any Richardson variety
$X_\la^\mu$ is $T$-convex, we have $C_{\tau,\sigma} \subset X_\la^\mu$ if and
only if $\mu \subset \tau \subset \sigma \subset \la$.

\begin{lemma}\label{lemma:chow1conn}%
  Let $X$ be a cominuscule flag variety, and let $\mu \subset \la \subset \cP_X$
  be straight shapes. Assume that $\theta = \la/\mu$ is connected and
  $\dex(\theta) = \dex(\cP_X)$. Then pushforward along the inclusion $X_\la^\mu
  \subset X$ is an isomorphism of groups $A_1(X_\la^\mu) \cong A_1(X)$.
\end{lemma}
\begin{proof}
  It follows from \Lemma{comparable-boxes}(c) that $\cP_X$ is connected, so we
  have $A_1(X) \cong \Z$ by \Lemma{indecomp-comin}. In view of
  \Lemma{comin-degree}, it suffices to show that $A_1(X_\la^\mu) =
  \Z\,[X_{\{\al\}}]$ holds for any long box $\al \in \theta$.
  \Proposition{chow1gen} implies that $A_1(X_\la^\mu)$ is generated by the
  classes $[C_{\sigma,\tau}]$ for which $\sigma$ and $\tau$ are adjacent and
  $\mu \subset \tau \subset \sigma \subset \la$. If $|\sigma/\tau| \geq 2$, then
  \Proposition{comin-triangle} shows that there exists a straight shape $\ka$
  such that $\tau \subset \ka \subset \sigma$ and $\ka$ is adjacent to both
  $\tau$ and $\sigma$. \Corollary{triangle} then implies that
  $[C_{\tau,\sigma}]$ is a multiple of either $[C_{\tau,\ka}]$ or
  $[C_{\ka,\sigma}]$. It follows by induction on $|\sigma/\tau|$ that
  $A_1(X_\la^\mu)$ is generated by the classes $[C_{\tau,\sigma}]$ for which
  $\sigma/\tau=\{\al\}$ contains a single box. In this case $C_{\tau,\sigma} =
  X_\sigma^\tau$ is a 1-dimensional Richardson variety contained in $X_\la^\mu$,
  and we have $[C_{\tau,\sigma}] = [X_{\{\al\}}]$. If all boxes of $\theta$ are
  long, then the claim follows from \Lemma{neighbor-class} and the assumption
  that $\theta$ is connected, so assume that $\theta$ contains both long and
  short boxes. If $\al' \in \theta$ is any short box, then since $\theta$ is
  connected, there exists a path from $\al'$ to some long box $\al \in \theta$,
  in which case \Lemma{neighbor-class} implies that $[X_{\{\al'\}}] =
  2[X_{\{\al\}}]$. It follows that $A_1(X_\la^\mu)$ is generated by the classes
  $[X_{\{\al\}}]$ defined by all long boxes $\al \in \theta$. It remains to
  prove that these classes are all the same. Since $X$ is either an odd quadric
  or a Lagrangian Grassmannian, an inspection of \Table{labsets} shows that the
  long boxes in $\cP_X$ are linearly ordered. Let the long boxes in $\theta$ be
  $\{\al_1 < \al_2 < \dots < \al_\ell\}$. If $\al_i < \al_{i+1}$ is a covering
  in $\cP_X$, then \Lemma{neighbor-class} shows $[X_{\{\al_i\}}] =
  [X_{\{\al_{i+1}\}}] \in A_1(X_\la^\mu)$. Otherwise \Table{labsets} shows that
  $\dex(\al_i)=\dex(\al_{i+1})$, and we have $\al_i < \be < \al_{i+1}$ for a
  unique short box $\be \in \cP_X$, such that $\{\al_i,\be,\al_{i+1}\}$ is a
  skew shape in $\cP_X$. Choose straight shapes $\tau,\sigma \subset \cP_X$ such
  that $\mu \subset \tau \subset \sigma \subset \la$ and $\sigma/\tau =
  \{\al_i,\be,\al_{i+1}\}$, and set $S = \dex(\sigma/\tau)$ and $Y = G_S
  w_\sigma.P_X$. We deduce from \Proposition{comin} that $Y = X_\sigma^\tau
  \cong Q^3$, hence $[X_{\{\al_i\}}] = [X_{\{\al_{i+1}\}}]$ holds in
  $A_1(X_\sigma^\tau)$ by \Lemma{comin-degree}. Since $X_\sigma^\tau \subset
  X_\la^\mu$, we obtain $[X_{\{\al_i\}}] = [X_{\{\al_{i+1}\}}]$ in
  $A_1(X_\la^\mu)$, which completes the proof.
\end{proof}

\begin{thm}\label{thm:chow1rich}%
  Let $X$ be a cominuscule flag variety, and let $\theta \subset \cP_X$ be a
  skew shape. Let $\theta_1, \dots, \theta_r$ be the connected components of
  $\theta$.\smallskip

  \titem{a} The Richardson variety $X_\theta$ is indecomposable if and only if
  $\theta$ is connected.\smallskip

  \titem{b} We have $X_\theta \cong X_{\theta_1} \times \dots \times
  X_{\theta_r}$ as algebraic varieties; the associated projections $X_\theta \to
  X_{\theta_i}$ are uniquely determined by the isomorphism class of $X_\theta$
  as a variety, up to reordering and automorphisms of the targets.\smallskip

  \titem{c} The pushforwards along the projections define an isomorphism of
  first Chow groups $A_1(X_\theta) \cong A_1(X_{\theta_1}) \oplus \dots \oplus
  A_1(X_{\theta_r})$.\smallskip

  \titem{d} The first Chow group $A_1(X_\theta)$ is a free Abelian group of rank
  $r$.\smallskip

  \titem{e} Assume that $X$ is the minimal embedding of $X_\theta$. Then proper
  pushforward along the inclusion $i : X_\theta \to X$ is an isomorphism of
  groups $i_* : A_1(X_\theta) \xrightarrow{\cong\,} A_1(X)$.
\end{thm}
\begin{proof}
  The isomorphism $X_\theta \cong X_{\theta_1} \times \dots \times X_{\theta_r}$
  follows from \Corollary{comin-rich-factor}, after which part (c) follows from
  \Lemma{chow1prod} and \Lemma{chowzero}. For the remaining claims we may assume
  that $X$ is the minimal embedding of $X_\theta$, so that $\dex(\theta) =
  \dex(\cP_X)$. If $\theta$ is connected, then so is $\cP_X$ by
  \Lemma{comparable-boxes}(c), so we have $A_1(X_\theta) \cong A_1(X) \cong \Z$
  by \Lemma{chow1conn} and \Lemma{indecomp-comin}. Part (a) now follows from
  \Lemma{indecomp}. The uniqueness of the factorization in part (b) follows from
  \Corollary{unique-factor}, and part (d) follows from (c) since
  $A_1(X_{\theta_i}) \cong \Z$ for $1 \leq i \leq r$. Finally, part (e) follows
  from part (c) and \Lemma{chow1conn}, noting that $\wh\theta_1, \dots,
  \wh\theta_r$ are the connected components of $\cP_X$. This completes the
  proof.
\end{proof}

\subsection{The Picard group of a cominuscule Richardson variety}

Given any complete variety $Z$, we let $A_1(Z)^*$ denote the dual of the first
Chow group of $Z$. The \emph{degree map} of $Z$ is the natural group
homomorphism $\Pic(Z) \to A_1(Z)^*$ defined by the intersection pairing
\[
  \Pic(Z) \otimes A_1(Z) \,\to\, \Z
  \ \ \ ; \ \ \
  \cL \otimes [C] \ \mapsto \int_C c_1(\cL) = \deg(\cL|_C) \,.
\]
Our computation of the Picard group $\Pic(X_\theta)$ requires the following
fact.

\begin{thm}\label{thm:comin-desing}%
  Let $X_\theta$ be any cominuscule Richardson variety. There exists a
  desingularization $\pi : \wt{X_\theta} \to X_\theta$ such that the degree
  map $\Pic(\wt{X_\theta}) \to A_1(\wt{X_\theta})^*$ is injective.
\end{thm}

We prove this statement in \Section{desing} when the minimal embedding of
$X_\theta$ is a product of cominuscule flag varieties of classical Lie types,
see \Corollary{desing-classical}. In response to our questions, the general case
of \Theorem{comin-desing} has been proved by David Anderson and Sam Payne
\cite{anderson.payne:chow}, by showing that Brion's \emph{brick resolution} from
the proof of \cite[Thm.~4.2.1]{brion:lectures} (see also
\cite[Appendix~A]{knutson.lam.ea:projections}) satisfies the required property
(and much more!).

\begin{thm}\label{thm:picard}%
  Let $X$ be a cominuscule flag variety, and let $\theta \subset \cP_X$ be a
  skew shape. Let $\theta_1, \dots, \theta_r$ be the connected components of
  $\theta$.%\smallskip

  \titem{a} The degree map is an isomorphism of groups $\Pic(X_\theta)
  \xrightarrow{\cong\,} A_1(X_\theta)^*$.\smallskip

  \titem{b} We have $\Pic(X_\theta) \cong \Pic(X_{\theta_1}) \oplus \dots \oplus
  \Pic(X_{\theta_r})$.\smallskip

  \titem{c} The Picard group $\Pic(X_\theta)$ is a free Abelian group of rank
  $r$.\smallskip

  \titem{d} Assume that $X$ is the minimal embedding of $X_\theta$. Then
  pullback along the inclusion $i : X_\theta \to X$ is an isomorphism of groups
  $i^*: \Pic(X) \xrightarrow{\cong\,} \Pic(X_\theta)$.
\end{thm}
\begin{proof}
  We may assume that $X$ is the minimal embedding of $X_\theta$, so that the
  pushforward map $i_* : A_1(X_\theta) \to A_1(X)$ is an isomorphism by
  \Theorem{chow1rich}(e). Let $\pi : \wt{X_\theta} \to X_\theta$ be the
  desingularization of \Theorem{comin-desing}. Since $X_\theta$ is normal, we
  have $\pi_*(\cO_{\wt{X_\theta}}) = \cO_{X_\theta}$ by Zariski's main theorem.
  Given any line bundle $\cL$ on $X_\theta$, we obtain
  \[
    \pi_*(\pi^* \cL) = \pi_*(\pi^*\cL \otimes \cO_{\wt{X_\theta}})
    = \cL \otimes \pi_*(\cO_{\wt{X_\theta}}) = \cL \,.
  \]
  It follows that the pullback map $\pi^* : \Pic(X_\theta) \to
  \Pic(\wt{X_\theta})$ is injective. We obtain the following commutative
  diagram, where the vertical maps are degree maps.
  \[
    \xymatrix{
      \Pic(X) \ar[r]_{i^*} \ar[d]_\cong &
      \Pic(X_\theta)\, \ar@{^{(}->}[r]_{\pi^*} \ar[d] &
      \Pic(\wt{X_\theta}) \ar@{^{(}->}[d]
      \\
      A_1(X)^* \ar[r]^{(i_*)^*}_\cong &
      A_1(X_\theta)^* \ar[r]^{(\pi_*)^*} &
      A_1(\wt{X_\theta})^*
    }
  \]
  The right square of this diagram implies that the degree map
  $\Pic(X_\theta) \to A_1(X_\theta)^*$ is injective, after which the left square
  shows that this map as well as the pullback map $i^* : \Pic(X) \to
  \Pic(X_\theta)$ are isomorphisms. The result therefore follows from
  \Theorem{chow1rich}.
\end{proof}

% !TeX root=comrich.tex

\section{Desingularization of Richardson varieties in
Grassmannians}\label{sec:desing}%

In this section we construct an explicit desingularization of any Richardson
variety in a Grassmannian of type A or a maximal isotropic Grassmannian, in the
style of Zelevinsky's desingularization of Schubert varieties
\cite{zelevinski:small}. This desingularization is obtained as a tower of
(isotropic) Grassmann bundles and is used to prove that \Theorem{comin-desing}
holds for products of cominuscule Richardson varieties of classical Lie types
when $\chr(\bK) \neq 2$. Our construction is of independent interest, as the
known desingularization of Richardson varieties in \cite{brion:positivity}
involves a fiber product of Bott--Samelson varieties, which is less explicit. We
give a unified treatment of the Grassmannian $\Gr(m,N)$ of type A, the
Lagrangian Grassmannian $\LG(m,2m)$ of type C, and the maximal orthogonal
Grassmannian $\OG(m,2m+1)$ of type B, noting that $\OG(m,2m+1)$ is isomorphic to
the cominuscule orthogonal Grassmannian $\OG(m+1,2m+2)$ of type D. The chosen
Grassmannian will be denoted $X$, and we use the construction of $X$ given in
\Section{classical}. In particular, we assume $\chr(\bK)\neq 2$ when $X$ is an
orthogonal Grassmannian, and we set $N = 2m$ for $X = \LG(m,2m)$, and $N=2m+1$
for $X = \OG(m,2m+1)$.

The partially ordered set $\cP_X$ is identified with a diagram of boxes in the
plane; see \Table{labsets}. For $X=\Gr(m,N)$, $\cP_X$ is a rectangle with $m$
rows and $N-m$ columns, and for $X=\LG(m,2m)$ and $X=\OG(m,2m+1)$, $\cP_X$ is
the upper-triangular part of an $m \times m$ rectangle. Given any box $\al \in
\cP_X$, let $\up(\al)$ be the row number of $\al$ counted from the top of
$\cP_X$ and let $\dn(\al)$ be the row number counted from the bottom, so that
$\up(\al)+\dn(\al) = m+1$. Let $\lf(\al)$ be the column number of $\al$ counted
from the left and let $\rt(\al)$ be the column number counted from the right
side of $\cP_X$, so that $\lf(\al)+\rt(\al) = N-m+1$ for $X=\Gr(m,N)$, and
$\lf(\al)+\rt(\al) = m+1$ for $X=\LG(m,2m)$ and for $X=\OG(m,2m+1)$. We also set
$\ur(\al) = \up(\al)+\rt(\al)-1$. The partial order on $\cP_X$ is given by $\al'
\leq \al''$ if and only if $\up(\al') \leq \up(\al'')$ and $\lf(\al') \leq
\lf(\al'')$.

\begin{center}

\begin{tikzpicture}[
  x=5mm,y=5mm,
  measurement/.style={{Stealth[scale=1]}-{Stealth[scale=1]}, draw=blue!50}
]
  \begin{scope}[local bounding box=schub]
  % Light gray internal grid lines: 6 rows by 9 columns
  \draw[step=1, lightgray!50, thin] (0,0) grid (9,6);

  % Black frame and label for box (row,column) = (3,4), counting rows from the top
  \draw[black, thick] (3,3) rectangle (4,4);
  \node at (3.5,3.5) {$\alpha$};

  % Measurement line from the left border to the right border of alpha
  \draw[measurement] (0,2.5) -- (4,2.5)
    node[midway, below, fill=white, inner sep=1pt] {$\lf(\alpha)$};

  % Measurement line from the bottom of alpha to the top border
  \draw[measurement] (2.5,3) -- (2.5,6)
    node[midway, left, fill=white, inner sep=1pt] {$\up(\alpha)$};

  % Measurement line from the left side of alpha to the right border
  \draw[measurement] (3,4.5) -- (9,4.5)
    node[midway, above, fill=white, inner sep=1pt] {$\rt(\alpha)$};

  % Measurement line from the top of alpha to the bottom border
  \draw[measurement] (4.5,4) -- (4.5,0)
    node[midway, right, fill=white, inner sep=1pt] {$\dn(\alpha)$};

  % Black outer outline
  \draw[black, thick] (0,0) rectangle (9,6);

  % Measurement line 1/2 unit under the rectangle indicating its width
  \draw[measurement] (0,-0.5) -- (9,-0.5)
    node[midway, below, fill=white, inner sep=1pt] {$N-m$};
  \end{scope}

  % Vertical measurement line between the two pictures indicating their common height
  \draw[measurement] (10.5,0) -- (10.5,6)
    node[midway, right, fill=white, inner sep=1pt] {$m$};

  \begin{scope}[shift={(12,0)}, local bounding box=staircase]
  % Dashed outline of the full 6 by 6 square containing the staircase.
  \draw[black, dashed] (0,0) rectangle (6,6);

  % Staircase made from the upper-right part of a 6 by 6 rectangle
  % Draw only the internal grid lines, leaving the border for the black outline.
  \foreach \x in {1,...,5} {
    \pgfmathtruncatemacro{\ymin}{6-\x}
    \draw[lightgray!50, thin] (\x,\ymin) -- (\x,6);
  }
  \foreach \y in {1,...,5} {
    \pgfmathtruncatemacro{\xmin}{6-\y}
    \draw[lightgray!50, thin] (\xmin,\y) -- (6,\y);
  }

  % Black frame and label for box (row,column) = (3,5), counting rows from the top
  \draw[black, thick] (4,3) rectangle (5,4);
  \node at (4.5,3.5) {$\alpha$};

  % Measurement line from the left border of the 6 by 6 rectangle to the right border of alpha
  \draw[measurement] (0,2.5) -- (5,2.5)
    node[pos=0.37, below, fill=white, inner sep=1pt] {$\lf(\alpha)$};

  % Measurement line from the bottom of alpha to the top border of the 6 by 6 rectangle
  \draw[measurement] (3.5,3) -- (3.5,6)
    node[midway, left, fill=white, inner sep=1pt] {$\up(\alpha)$};

  % Measurement line from the left side of alpha to the right border of the 6 by 6 rectangle
  \draw[measurement] (4,4.5) -- (6,4.5)
    node[midway, above, yshift=1.2pt, fill=white, inner sep=1pt] {$\rt(\alpha)$};

  % Measurement line from the top of alpha to the bottom border of the 6 by 6 rectangle
  \draw[measurement] (5.5,4) -- (5.5,0)
    node[pos=0.625, left, fill=white, inner sep=0pt] {\tiny$\dn(\alpha)$};

  % Black outer outline of the staircase
  \draw[black, thick]
    (5,0) -- (6,0) -- (6,6) -- (0,6) -- (0,5)
    -- (1,5) -- (1,4) -- (2,4) -- (2,3)
    -- (3,3) -- (3,2) -- (4,2) -- (4,1)
    -- (5,1) -- cycle;

  % Measurement line 1/2 unit under the 6 by 6 rectangle indicating its width
  \draw[measurement] (0,-0.5) -- (6,-0.5)
    node[midway, below, fill=white, inner sep=1pt] {$m$};
  \end{scope}
\end{tikzpicture}

\end{center}

Let $E_\bull \subset \bK^N$ be the $B$-stable flag defined by $E_i = \langle
e_1, \dots, e_i \rangle$ for $0 \leq i \leq N$, and let $F_\bull \subset \bK^N$
be the $B^-$-stable flag defined by $F_i = \langle e_{N+1-i}, \dots, e_N
\rangle$. In types B and C we have $E_i^\perp = E_{N-i}$ and $F_i^\perp =
F_{N-i}$, where $U^\perp$ denotes the kernel of the linear map $\bK^N \to U^*$
defined by the nonzero bilinear form $\om$ on $\bK^N$, for any subspace $U
\subset \bK^N$. Each box $\al \in \cP_X$ defines a $B$-stable Schubert variety
\[
  X_\al = \{ V \in X \mid \dim(V^\perp + E_{N-\ur(\al)}) \leq N-\rt(\al) \} \,,
\]
where we write $V^\perp = V$ in type A, and a $B^-$-stable Schubert variety
\[
  X^\al = \{ V \in X \mid \dim(V \cap F_{\ur(\al)}) \geq \up(\al) \} \,.
\]

\begin{prop}\label{prop:schubvar}%
  The Schubert varieties in $X$ defined by a straight shape $\la \subset \cP_X$
  are given by
  \[
    X_\la = \bigcap_{\al \notin \la} X_\al
    \text{ \ \ \ \ and \ \ \ \ }
    X^\la = \bigcap_{\al \in \la} X^\al \,.
  \]
  The first intersection is over all boxes of $\cP_X \ssm \la$; it suffices to
  include all minimal boxes of $\cP_X \ssm \la$ in this intersection. Similarly,
  it suffices to include all maximal boxes of $\la$ in the second intersection.
\end{prop}
\begin{proof}
  The definition of $X^\la$ matches known definitions of Schubert varieties,
  noting that all Schubert conditions follow from those given by maximal boxes
  of $\la$; see \cite[\S9.4]{fulton:young},
  \cite[\S3.1]{fulton.pragacz:schubert}, and
  \cite[Prop.~A1]{buch.kresch.ea:giambelli}. Let $\la^\vee = w_0^X.(\cP_X \ssm
  \la)$ be the Poincaré dual shape of $\la$; see
  \cite[\S2]{buch.samuel:k-theory}. Using that $X_{\la^\vee} = w_0.X^\la$, it
  suffices to show that $X_\al = w_0.X^{w_0^X.\al}$ for all $\al \in \cP_X$. The
  effect of applying $w_0$ is to replace $F_\bull$ with $E_\bull = w_0.F_\bull$,
  and the action of $w_0^X$ on $\cP_X$ is the order-reversing involution that
  rotates by 180 degrees in type A and reflects across a southwest-to-northeast
  diagonal in types B and C. In type A we have $\ur(w_0^X.\al) = N-\ur(\al)$,
  $\up(w_0^X.\al) = \dn(\al)$, and $m+N-\ur(\al)-\dn(\al) = N-\rt(\al)$, which
  implies
  \[
    \begin{split}
      w_0.X^{w_0^X.\al}
      &= \{ V \in X \mid \dim(V \cap E_{N-\ur(\al)}) \geq \dn(\al) \} \\
      &= \{ V \in X \mid \dim(V + E_{N-\ur(\al)}) \leq N-\rt(\al) \} \,.
    \end{split}
  \]
  In types B and C we have $(V \cap E_{\ur(\al)})^\perp = V^\perp +
  E_{N-\ur(\al)}$, $\ur(w_0^X.\al) = \ur(\al)$, and $\up(w_0^X.\al) = \rt(\al)$,
  which implies
  \[
    \begin{split}
    w_0.X^{w_0^X.\al}
    &= \{ V \in X \mid \dim(V \cap E_{\ur(\al)}) \geq \rt(\al) \} \\
    &= \{ V \in X \mid \dim(V^\perp + E_{N-\ur(\al)}) \leq N - \rt(\al) \} \,.
    \end{split}
  \]
  The proposition follows from this.
\end{proof}

Given a Richardson variety $X_\la^\mu \subset X$, we define the sets of integers
\[
  a^-(\mu) = \{ \up(\al) : \al \in \max(\mu) \}
  \text{ \ \ \ and \ \ \ }
  a^+(\la) = \{ N-\rt(\al) : \al \in \min(\cP_X\ssm\la) \} \,.
\]
Let $\Fl(a^-(\mu) \cup a^+(\la); N)$ be the variety of partial flags in $\bK^N$
of the indicated dimensions; this flag variety has Lie type A regardless of the
type of $X$. Let $\wh{X}_\la^\mu \subset \Fl(a^-(\mu) \cup a^+(\la); N)$ be the
Richardson variety of partial flags $V_\bull$ satisfying
\[
  \begin{split}
  V_{N-\rt(\al)} \supset E_{N-\ur(\al)}
  &\text{ \ \ \ for each minimal box $\al$ in $\cP_X \ssm \la$, and} \\
  V_{\up(\al)} \subset F_{\ur(\al)}
  &\text{ \ \ \ for each maximal box $\al$ in $\mu$.}
  \end{split}
\]
Let $\wt{X}_\la^\mu \subset \Fl(a^-(\mu) \cup \{m\} \cup a^+(\la); N)$ be the
subvariety of partial flags $V_\bull$ satisfying the above conditions defining
$\wh{X}_\la^\mu$, in addition to $V_m \in X$; for $X = \OG(m,2m+1)$ we further
require $V_m^\perp \subset V_{N-\rt(\al)}$ for $\al \in \min(\cP_X\ssm \la)$. We
prove in \Theorem{desing-grass} below that the projection $\wt{X}_\la^\mu \to
X_\la^\mu$ defined by $V_\bull \mapsto V_m$ is a resolution of singularities.

\begin{example}\label{example:resol1}%
  Let $X = \LG(9,18)$ and consider the Richardson variety $X_\la^\mu$ defined by
  the skew shape:
  \begin{center}
  \begin{tikzpicture}
    [x=.33cm,y=.33cm]
    \newcommand{\crn}{ -- ++(-1,0) -- ++(0,1)}
    \newcommand{\bxx}{ -- ++(0,1) -- ++(1,0) -- ++(0,-1) -- cycle}
    \begin{scope}
      \draw [lightgray!50, thin] (1,9) -- (1,8) -- (9,8);
      \draw [lightgray!50, thin] (2,9) -- (2,7) -- (9,7);
      \draw [lightgray!50, thin] (3,9) -- (3,6) -- (9,6);
      \draw [lightgray!50, thin] (4,9) -- (4,5) -- (9,5);
      \draw [lightgray!50, thin] (5,9) -- (5,4) -- (9,4);
      \draw [lightgray!50, thin] (6,9) -- (6,3) -- (9,3);
      \draw [lightgray!50, thin] (7,9) -- (7,2) -- (9,2);
      \draw [lightgray!50, thin] (8,9) -- (8,1) -- (9,1);
      \draw (9,9) -- (9,0) \crn\crn\crn\crn\crn\crn\crn\crn\crn -- cycle;
      \draw [very thick] (6,3) -- ++(2,0) -- ++(0,1) -- ++(1,0) -- ++(0,5)
      -- ++(-2,0) -- ++(0,-2) -- ++(-2,0) -- ++(0,-1) -- ++(-2,0);
      \node at (6.5,5) {$\la/\mu$};
    \end{scope}
  \end{tikzpicture}
  \end{center}
  Since $V^\perp = V$ holds for $V \in X$, we obtain
  \[
    \begin{split}
      X^\mu &= \{V \in X \mid \dim(V\cap F_4) \geq 2
      \text{ and } \dim(V\cap F_7) \geq 3\} \text{ \ \ and }
      \\
      X_\la &= \{V \in X \mid \dim(V+E_9)\leq 15
      \text{ and } \dim(V+E_{12})\leq 17\} \,.
    \end{split}
  \]
  The Richardson variety $\wh{X}_\la^\mu \subset \Fl(2,3,15,17;18)$ is defined
  by the conditions
  \[
    V_2 \subset F_4\,, \ \ V_3 \subset F_7\,, \ \ V_{15} \supset E_9\,, \ \
    V_{17} \supset E_{12} \,.
  \]
\end{example}

Using that $\wh{X}_\la^\mu$ is an intersection of non-singular opposite Schubert
varieties in $\Fl(a^-(\mu)\cup a^+(\la);N)$, it follows from Kleiman's
transversality theorem \cite{kleiman:transversality} that $\wh{X}_\la^\mu$ is
non-singular in characteristic zero. The same applies to $\wt{X}_\la^\mu$ in
type A. To prove that $\wh{X}_\la^\mu$ is non-singular in arbitrary
characteristic, we show that this variety can be constructed as a tower of
Grassmann bundles.

Choose $\al \in \max(\mu) \cup \min(\cP_X\ssm\la)$ with $\ur(\al)$ maximal. Let
$\mu' \subset \cP_X$ and $\la' \subset \cP_X$ be the unique straight shapes
defined by $\max(\mu') = \max(\mu) \ssm \{\al\}$ and $\min(\cP_X\ssm\la') =
\min(\cP_X\ssm\la) \ssm \{\al\}$. We have $\mu'=\mu$ or $\la'=\la$.

\begin{example}\label{example:resol2}%
  For the Richardson variety of \Example{resol1} we obtain the following box
  $\al$ and straight shapes $\mu' \subset \la' \subset \cP_X$. The boxes $\al'$
  and $\al''$ are used in the proof of \Lemma{Xhat-bundle}, and the shapes
  $\cP_Y$ and $\cP_Z$ are discussed in \Remark{resol-fibers}.\smallskip
  \begin{center}
  \begin{tikzpicture}[x=.33cm,y=.33cm]
    \newcommand{\crn}{ -- ++(-1,0) -- ++(0,1)}
    \newcommand{\bxx}{ -- ++(0,1) -- ++(1,0) -- ++(0,-1) -- cycle}
    \begin{scope}
      \draw (9,9) -- (9,0) \crn\crn\crn\crn\crn\crn\crn\crn\crn -- cycle;
      \draw [gray] (6,2) \bxx;
      \node at (6.5,2.5) {$\al$};
      \draw [very thick] (6,3) -- ++(2,0) -- ++(0,1) -- ++(1,0) -- ++(0,5)
      -- ++(-2,0) -- ++(0,-2) -- ++(-2,0) -- ++(0,-1) -- ++(-2,0);
      \node at (6.5,5) {$\la/\mu$};
    \end{scope}
    \begin{scope}[shift={(12,0)}]
      \draw (9,9) -- (9,0) \crn\crn\crn\crn\crn\crn\crn\crn\crn -- cycle;
      \draw [gray] (4,6) \bxx;
      \node at (4.5,6.55) {$\scriptstyle\al'$};
      \draw [gray] (8,3) \bxx;
      \node at (8.55,3.55) {$\scriptstyle{\al\!}''$};
      \draw [very thick] (8,1) -- ++(0,3) -- ++(1,0) -- ++(0,5)
      -- ++(-2,0) -- ++(0,-2) -- ++(-2,0) -- ++(0,-1) -- ++(-2,0);
      \node at (6.5,5) {$\la'/\mu'$};
    \end{scope}
    \begin{scope}[shift={(24,0)}]
      \draw (9,9) -- (9,0) \crn\crn\crn\crn\crn\crn\crn\crn\crn -- cycle;
      \draw [gray] (6,3) -- ++(2,0) -- ++(0,1) -- ++(1,0) -- ++(0,5)
      -- ++(-2,0) -- ++(0,-2) -- ++(-2,0) -- ++(0,-1) -- ++(-2,0);
      \draw [very thick] (5,6) -- ++(0,1) -- ++(2,0) -- ++(0,2) -- ++(2,0) --
      ++(0,-5) -- ++(-1,0) -- ++(0,-1) -- ++(-3,0) -- ++(0,1) \crn\crn -- cycle;
      \draw [very thick] (6,3) -- ++(0,3) -- ++(-3,0);
      \draw [very thick] (8,3) -- ++(0,3) -- ++(-2,0);
      \draw [very thick] (8,6) -- ++(0,1) -- ++(-2,0);
      \draw [very thick] (8,7) -- ++(1,0);
      \node at (7,4.8) {$\cP_Y$};
      \node at (5,4.8) {$\cP_Z$};
    \end{scope}
  \end{tikzpicture}
  \end{center}
\end{example}

\begin{lemma}\label{lemma:Xhat-bundle}%
  The natural projection $\wh{X}_\la^\mu \to \wh{X}_{\la'}^{\mu'}$ is a
  Grassmann bundle.
\end{lemma}
\begin{proof}
  Choose $\al' \in \max(\mu')$ with $\ur(\al')$ maximal, and choose $\al'' \in
  \min(\cP_X \ssm \la')$ with $\ur(\al'')$ maximal. If $\mu'$ is empty, so that
  $\al'$ does not exist, we use the notation $\up(\al') = \ur(\al') = 0$.
  Similarly, we write $\rt(\al'') = \ur(\al'') = 0$ when $\la' = \cP_X$. Let
  $\cV_\bull$ denote the tautological flag on $\wh{X}_{\la'}^{\mu'}$.

  Assume first that $\al \in \mu$. Given a point $V_\bull \in
  \wh{X}_{\la'}^{\mu'} = \wh{X}_\la^{\mu'}$, points in the fiber over $V_\bull$
  are obtained by adjoining a subspace $V_{\up(\al)}$ of dimension $\up(\al)$ to
  $V_\bull$ satisfying $V_{\up(\al')} \subset V_{\up(\al)} \subset
  V_{N-\rt(\al'')} \cap F_{\ur(\al)}$. By our choice of $\al$, we have
  $V_{N-\rt(\al'')} + F_{\ur(\al)} \supset E_{N-\ur(\al'')} + F_{\ur(\al)} =
  \bK^N$, which implies that $\cV_{N-\rt(\al'')} \cap F_{\ur(\al)}$ is a
  well-defined vector subbundle of $\wh{X}_\la^{\mu'} \times \bK^N$ of rank
  $\ur(\al)-\rt(\al'')$. We obtain
  \[
    \wh{X}_\la^\mu \cong
    \Gr\left( \up(\al)-\up(\al'),
    (\cV_{N-\rt(\al'')} \cap F_{\ur(\al)} ) / \cV_{\up(\al')} \right)
    \to \wh{X}_\la^{\mu'} \,.
  \]

  Otherwise we have $\al \in \cP_X \ssm \la$, and points in the fiber over
  $V_\bull \in \wh{X}_{\la'}^{\mu'} = \wh{X}_{\la'}^\mu$ are obtained by
  adjoining a subspace $V_{N-\rt(\al)}$ of dimension $N-\rt(\al)$ to $V_\bull$
  satisfying $V_{\up(\al')} + E_{N-\ur(\al)} \subset V_{N-\rt(\al)} \subset
  V_{N-\rt(\al'')}$. By our choice of $\al$, we have $V_{\up(\al')} \cap
  E_{N-\ur(\al)} \subset F_{\ur(\al')} \cap E_{N-\ur(\al)} = 0$, which implies
  that $\cV_{\up(\al')} \oplus E_{N-\ur(\al)}$ is a vector subbundle of
  $\wh{X}_{\la'}^\mu \times \bK^N$. Using that $\ur(\al)-\rt(\al) = \up(\al)-1$,
  we obtain
  \[
    \wh{X}_\la^\mu \cong
    \Gr\left( \up(\al)-\up(\al')-1,
    \cV_{N-\rt(\al'')} / (\cV_{\up(\al')} \oplus E_{N-\ur(\al)}) \right)
    \to \wh{X}_{\la'}^\mu \,.
  \]
  This completes the proof.
\end{proof}

\begin{remark}\label{remark:resol-fibers}%
  \def\resol#1{[#1]}%
  Let $\resol{\la/\mu}$ denote the subset of boxes in $\la/\mu$ that are in the
  same row as a box from $\mu$ or in the same column as a box from
  $\cP_X\ssm\la$. It follows from the proof of \Lemma{Xhat-bundle} that
  $\dim(\wh{X}_\la^\mu) = \# \resol{\la/\mu}$. In fact, $\resol{\la/\mu} \ssm
  \resol{\la'/\mu'}$ is a rectangle that can be identified with $\cP_Y$, where
  $Y$ is the fiber of the projection $\wh{X}_\la^\mu \to \wh{X}_{\la'}^{\mu'}$.
  Moreover, the shape $\la/\mu \ssm \resol{\la/\mu}$ is either a rectangle or
  an upper-triangular shape. The proof of \Lemma{Xtilde-bundle} shows that this
  shape can be identified with $\cP_Z$, where $Z$ is the fiber of the projection
  $\wt{X}_\la^\mu \to \wh{X}_\la^\mu$. The desingularization $\wt{X}_\la^\mu$ of
  the Richardson variety from \Example{resol1} is a tower of four Grassmann
  bundles followed by one Lagrangian Grassmannian bundle, where the fibers
  correspond to the rectangular and upper-triangular shapes in the last diagram
  of \Example{resol2}.
\end{remark}

\begin{lemma}\label{lemma:Xtilde-bundle}%
  The natural projection $\wt{X}_\la^\mu \to \wh{X}_\la^\mu$ is a Grassmann
  bundle when $X = \Gr(m,N)$ has type A, is a Lagrangian Grassmannian bundle
  when $X = \LG(m,2m)$ has type C, and is a locally trivial bundle of maximal
  orthogonal Grassmannians when $X = \OG(m,2m+1)$ has type B.
\end{lemma}
\begin{proof}
  Choose $\al' \in \max(\mu)$ with $\ur(\al')$ maximal, and choose $\al'' \in
  \min(\cP_X \ssm \la)$ with $\ur(\al'')$ maximal. Let $\cV_\bull$ be the
  tautological flag on $\wh{X}_\la^\mu$. Points in the fiber over $V_\bull \in
  \wh{X}_\la^\mu$ are obtained by adjoining a subspace $V_m$ of dimension $m$ to
  $V_\bull$ satisfying $V_{\up(\al')} \subset V_m \subset V_m^\perp \subset
  V_{N-\rt(\al'')}$. In type A we obtain
  \[
    \wt{X}_\la^\mu \cong \Gr\left(m-\up(\al'),
    \cV_{N-\rt(\al'')}/\cV_{\up(\al')}\right) \to \wh{X}_\la^\mu \,.
  \]

  Assume that $X$ has type B or C. For $V_\bull \in \wh{X}_\la^\mu$ we have
  $V_{\up(\al')} \subset F_{\ur(\al')}$ and $V_{N-\rt(\al'')}^\perp \subset
  E_{N-\ur(\al'')}^\perp = E_{\ur(\al'')}$. It follows that $V_{\up(\al')}$ and
  $V_{N-\rt(\al'')}^\perp$ are isotropic subspaces of $\bK^N$, and also that
  $V_{\up(\al')} \cap V_{N-\rt(\al'')}^\perp = 0$. Therefore, $V_{\up(\al')} +
  V_{N-\rt(\al'')}^\perp = V_{\up(\al')} \oplus V_{N-\rt(\al'')}^\perp$ is a
  direct sum in $\bK^N$. Since $\up(\al') \leq m \leq N-\rt(\al'')$, we obtain
  $V_{\up(\al')} \subset V_{N-\rt(\al'')}$ and $V_{N-\rt(\al'')}^\perp \subset
  V_{\up(\al')}^\perp$. It follows that $V_{\up(\al')} + V_{N-\rt(\al'')}^\perp
  \subset V_{\up(\al')}^\perp \cap V_{N-\rt(\al'')} = (V_{\up(\al')} +
  V_{N-\rt(\al'')}^\perp)^\perp$. The form on $\bK^N$ therefore defines a
  non-degenerate bilinear form on the vector bundle
  \[
    \cH = (\cV_{\up(\al')}^\perp \cap \cV_{N-\rt(\al'')}) /
    (\cV_{\up(\al')} + \cV_{N-\rt(\al'')}^\perp) \,.
  \]
  When $X = \LG(m,2m)$ has type C, $\wt{X}_\la^\mu$ is the Lagrangian
  Grassmannian bundle
  \[
    \wt{X}_\la^\mu = \LG(r, \cH) \to \wh{X}_\la^\mu \,,
  \]
  where $r = m-\up(\al')-\rt(\al'')$. This bundle is automatically locally
  trivial, as any symplectic vector bundle has a local basis $\{b_1, \dots,
  b_{2r}\}$ satisfying $\om(b_i,b_j) = \delta_{i+j,2r+1}$ for $1 \leq i \leq j
  \leq 2r$.

  Assume $X = \OG(m,2m+1)$ has type B. Using that $F_{m+1} \subset
  F_{\ur(\al')}^\perp \subset V_{\up(\al')}^\perp$ and $E_{m+1} \subset
  E_{N-\ur(\al'')} \subset V_{N-\rt(\al'')}$, it follows that
  \[
    V_{\up(\al')}^\perp
    \cap V_{N-\rt(\al'')} = (E_m \cap V_{\up(\al')}^\perp) \oplus \bK e_{m+1}
    \oplus (F_m \cap V_{N-\rt(\al'')}) \,,
  \]
  hence $\cH$ factors as the direct sum
  \[
    \cH =
    (E_m \cap \cV_{\up(\al')}^\perp)/\cV_{N-\rt(\al'')}^\perp
    \oplus \bK e_{m+1} \oplus
    (F_m \cap \cV_{N-\rt(\al'')})/\cV_{\up(\al')} \,.
  \]
  Each point $V_\bull \in \wh{X}_\la^\mu$ therefore has an open neighborhood
  on which the orthogonal bundle $\cH$ can be trivialized; that is, $\cH$ has a
  local basis of sections $\{b_1, \dots, b_{2r+1}\}$ satisfying $\om(b_i,b_j) =
  \delta_{i+j,2r+2}$ for $1 \leq i, j \leq 2r+1$, where $r =
  m-\up(\al')-\rt(\al'')$. We deduce that $\wt{X}_\la^\mu \to \wh{X}_\la^\mu$ is
  a locally trivial bundle with fiber $\OG(r,2r+1)$.
\end{proof}

\begin{thm}\label{thm:desing-grass}%
  The map $\wt{X}_\la^\mu \to X_\la^\mu$ defined by $V_\bull \mapsto V_m$ is a
  desingularization.
\end{thm}
\begin{proof}
  It follows from \Lemma{Xhat-bundle} and \Lemma{Xtilde-bundle} that
  $\wt{X}_\la^\mu$ is non-singular and irreducible. We must show that the
  projection $\pi : \wt{X}_\la^\mu \to X_\la^\mu$ is birational. The definition
  of $\wt{X}_\la^\mu$ implies that $V_m \in X_\la^\mu$ for each $V_\bull \in
  \wt{X}_\la^\mu$, so $\pi$ is a well-defined morphism of varieties. For each
  point $V$ in the open Richardson variety $\oX_\la^\mu = B w_\la.P_X \cap B^-
  w_\mu.P_X$, we have $\dim(V \cap F_{\ur(\al)}) = \up(\al)$ for $\al \in
  \max(\mu)$, and $\dim(V^\perp + E_{N-\ur(\al)}) = N-\rt(\al)$ for $\al \in
  \min(\cP_X \ssm \la)$. We therefore obtain a point $V_\bull \in
  \wt{X}_\la^\mu$ by setting $V_{\up(\al)} = V \cap F_{\ur(\al)}$ for $\al \in
  \max(\mu)$, $V_m = V$, and $V_{N-\rt(\al)} = V^\perp + E_{N-\ur(\al)}$ for
  $\al \in \min(\cP_X \ssm \la)$. The assignment $V \mapsto V_\bull$ defines a
  morphism $\oX_\la^\mu \to \wt{X}_\la^\mu$ that is inverse to $\pi$.
\end{proof}

When $X$ has type A or C, it follows immediately from our construction that the
degree map $\Pic(\wt{X}_\la^\mu) \to A_1(\wt{X}_\la^\mu)^*$ is injective,
thanks to the explicit descriptions of the Chow rings of Grassmann bundles
\cite[Prop.~14.6.3]{fulton:intersection} and Lagrangian Grassmannian bundles
\cite[Prop.~13.4.1]{anderson.fulton:equivariant}. Analogous results for bundles
of orthogonal Grassmannians seem to be missing from the literature (see
\cite{edidin.graham:characteristic} for some related results). To establish
\Theorem{comin-desing} for cominuscule Richardson varieties of classical Lie
types, we use the following result.

\begin{prop}\label{prop:int-inj}%
  Let $\rho : Z \to Y$ be a locally trivial fibration of non-singular projective
  varieties with fiber $F$. Assume that $Y$ is rational and the degree maps
  $\Pic(Y) \to A_1(Y)^*$ and $\Pic(F) \to A_1(F)^*$ are injective. Then the
  degree map $\Pic(Z) \to A_1(Z)^*$ is also injective. Moreover, we have an
  exact sequence
  \[
    0 \to \Pic(Y) \xrightarrow{\ \rho^*} \Pic(Z) \xrightarrow{\ i^*}
    \Pic(F) \to 0 \,,
  \]
  where $i : F = \rho^{-1}(y_0) \subset Z$ is the inclusion of the fiber over a
  general point $y_0 \in Y$.
\end{prop}
\begin{proof}
  Let $U \subset Y$ be an open neighborhood of $y_0$ such that $\rho^{-1}(U)
  \cong U \times F$, with $\rho$ identified with the first projection. Since $Y$
  is rational, we may assume that $U$ is isomorphic to an open subset of affine
  space. Set $d = \dim(Y)$ and $e = \dim(Z)$. The rows of the following
  commutative diagram are exact by \cite[Prop.~1.8]{fulton:intersection}.
  \[
    \xymatrix{
      A_{e-1}(\rho^{-1}(Y\ssm U)) \ar[r] &
      \Pic(Z) \ar[r] & \Pic(\rho^{-1}(U)) \ar[r] & 0
      \\
      A_{d-1}(Y\ssm U) \ar[u]^{\rho^*}_{\cong} \ar[r] &
      \Pic(Y) \ar[u]^{\rho^*} \ar[r] & \Pic(U) = 0
    }
  \]
  This diagram implies that the image of $\rho^* : \Pic(Y) \to \Pic(Z)$ is the
  kernel of the restriction map $\Pic(Z) \to \Pic(\rho^{-1}(U))$. Using that
  $i^* : \Pic(\rho^{-1}(U)) \to \Pic(F)$ is an isomorphism, we deduce that the
  top row of the following commutative diagram is right exact.
  \[
    \xymatrix{
      0 \ar[r] &
      \Pic(Y) \ar[r]^{\rho^*} \ar@{^{(}->}[d] &
      \Pic(Z)\, \ar[r]^{i^*} \ar[d] &
      \Pic(F) \ar@{^{(}->}[d] \ar[r] & 0
      \\
      &
      A_1(Y)^* \ar[r]^{(\rho_*)^*}_\subset &
      A_1(Z)^* \ar[r]^{(i_*)^*} &
      A_1(F)^*
    }
  \]
  Since $\rho_* : A_1(Z) \to A_1(Y)$ is surjective, it follows that $(\rho_*)^*
  : A_1(Y)^* \to A_1(Z)^*$ and $\rho^* : \Pic(Y) \to \Pic(Z)$ are injective.
  This establishes the short exact sequence of the proposition. Finally, the
  injectivity of the map $\Pic(Z) \to A_1(Z)^*$ follows from the second diagram.
\end{proof}

\ignore{
\begin{remark}
  Let $\emptyset \neq U \subset \bA^n$ be open, and let $F$ be an irreducible
  variety of dimension $d$. The regular embedding $i : F \to U \times F$ and the
  open inclusion $j : U \times F \to \bA^n \times F$ define group homomorphisms
  $i^* : A_{n+d-1}(F \times U) \to A_{d-1}(F)$ and $j^* : A_{n+d-1}(F \times
  \bA^n) \to A_{n+d-1}(F \times U)$. Since $i^* j^*$ is an isomorphism and $j^*$
  is surjective, it follows that $i^*$ is an isomorphism.
\end{remark}
}

\begin{cor}\label{cor:desing-classical}%
  Let $X_\theta$ be a cominuscule Richardson variety whose indecomposable
  factors have classical Lie types A, B, C, and D. There exists a
  desingularization $\wt{X_\theta} \to X_\theta$ such that $\wt{X_\theta}$ is a
  tower of locally trivial bundles, where the fibers are Grassmannians of type
  A, Lagrangian Grassmannians, or maximal orthogonal Grassmannians. Moreover,
  the degree map $\Pic(\wt{X_\theta}) \to A_1(\wt{X_\theta})^*$ is
  injective.
\end{cor}
\begin{proof}
  If $\wt{X_\theta}$ is any variety constructed as a tower of locally trivial
  bundles of flag varieties, then \Proposition{int-inj} shows that the
  degree map $\Pic(\wt{X_\theta}) \to A_1(\wt{X_\theta})^*$ is injective.
  To see that the required desingularization exists, we may assume that
  $X_\theta$ is indecomposable with minimal embedding $X$. If $X$ is a quadric
  hypersurface, then $X_\theta$ is a Schubert variety in $X$, and the
  Bott--Samelson desingularization is a tower of $\bP^1$-bundles. Otherwise we
  use the desingularization of \Theorem{desing-grass}.
\end{proof}

% !TeX root=comrich.tex

\section{Hilbert Polynomials}\label{sec:hilbert}

The determination of the Picard group of a cominuscule Richardson variety allows
us to define additional geometric invariants. A first example is the Hilbert
polynomial of the minimal ample line bundle, which we study in this section.

Let $X = G/P_X$ be a cominuscule flag variety, let $\theta \subset \cP_X$ be a
skew shape, and let $X_\theta \subset X$ be the associated cominuscule
Richardson variety. It follows from \Theorem{picard} that the Picard group
$\Pic(X_\theta)$ contains a unique minimal ample line bundle, which we denote by
$\cO_{X_\theta}(1)$. Let $H_\theta(r) \in \Q[r]$ be the associated Hilbert
polynomial defined by
\[
  H_\theta(r) = \chi(X_\theta, \cO_{X_\theta}(r)) =
  \sum_{i \geq 0} (-1)^i \dim H^i(X_\theta,\cO_{X_\theta}(r)) \,,
\]
where $\cO_{X_\theta}(r) = \cO_{X_\theta}(1)^{\otimes r}$. The Hilbert
polynomial $H_\theta(r)$ depends only on the isomorphism class of $X_\theta$ as
an algebraic variety.

More generally, let $H^X_\theta(r) \in \Q[r]$ denote the Hilbert polynomial
defined by
\[
  H^X_\theta(r) = \chi(X_\theta, \cO_X(r)|_{X_\theta}) \,,
\]
where $\cO_X(1)$ is the minimal ample line bundle in $\Pic(X)$. This polynomial
depends on the embedding $X_\theta \subset X$, but has the advantage that it
satisfies a recursive formula, which we prove in \Corollary{hilbert} below. When
$X$ is the minimal embedding of $X_\theta$, we have $\cO_{X_\theta}(1) =
\cO_X(1)|_{X_\theta}$ and $H_\theta(r) = H^X_\theta(r)$ by \Theorem{picard}(d),
so our formula for $H^X_\theta(r)$ also covers $H_\theta(r)$.

Let $K(X)$ be the $K$-theory ring of $X$. Given straight shapes $\la,\mu \subset
\cP_X$, we denote the corresponding Schubert classes in $K(X)$ by $\cO_\la =
[\cO_{X_\la}]$ and $\cO^\mu = [\cO_{X^\mu}]$. Let $\chi : K(X) \to \Z$ be the
sheaf Euler characteristic map. We have
\[
  \chi(\cO_\la \cdot \cO^\mu) = \chi([\cO_{X_\la \cap X^\mu}]) = \begin{cases}
    1 & \text{if $\mu \subset \la$,} \\
    0 & \text{otherwise.}
  \end{cases}
\]

A skew shape $\theta \subset \cP_X$ will be called \emph{short} if it consists
of short boxes. The skew shape $\theta$ is called a \emph{rook strip} if no two
boxes of $\theta$ are comparable. A \emph{short rook strip} is a rook strip
consisting of short boxes. Given straight shapes $\la, \nu \subset \cP_X$,
define the constant
\[
  C_\nu^\la(X) = \# \{ \mu \subset \cP_X \mid \la \subset \mu \subset \nu
  \text{ and } \mu/\la \text{ is a short skew shape} \} \,.
\]
We have $C_\nu^\nu(X) = 1$, and $C_\nu^\la(X)$ is nonzero only if $\la \subset
\nu$. When $\Phi$ is simply-laced, we have $C_\nu^\la(X)=1$ whenever $\la \subset
\nu$.

The following statement is a special case of the Pieri--Chevalley formula from
\cite{fulton.lascoux:pieri,pittie.ram:pieri-chevalley}. We give an expository
proof based on identities from \cite{buch.chaput.ea:chevalley}.

\begin{lemma}\label{lemma:multO1}%
  Multiplication by $[\cO_X(1)]$ in $K(X)$ is given by
  \[
    [\cO_X(1)] \cdot \cO_\nu = \sum_{\la \subset \nu} C_\nu^\la(X)\, \cO_\la \,,
  \]
  where the sum is over all straight shapes $\la$ contained in $\nu$.
\end{lemma}
\begin{proof}
  Define linear endomorphisms $\phi$ and $\psi$ of $K(X)$ by
  \[
    \phi(\cO^\la) = \sum_{\mu/\la \text{ short rook strip}}
    (-1)^{|\mu/\la|}\, \cO^\mu \,,
  \]
  where the sum is over all $\mu \subset \cP_X$ such that $\la \subset \mu$ and
  $\mu/\la$ is a short rook strip, and
  \[
    \psi(\cO^\mu) = \sum_{\nu/\mu \text{ rook strip}}
      (-1)^{|\nu/\mu|}\, \cO^\nu \,,
  \]
  the sum over all $\nu \subset \cP_X$ such that $\mu \subset \nu$ and $\nu/\mu$
  is any rook strip. By \cite[Lemma~3.5 and
  Theorem~3.6]{buch.chaput.ea:chevalley} we have
  \[
    \chi(\psi(\cO^\la) \cdot \cO_\nu) = \delta_{\la,\nu}
    \text{ \ \ \ and \ \ \ }
    [\cO_X(-1)] \cdot \cO^\la = \psi(\phi(\cO^\la)) \,.
  \]
  Notice that $\phi$ is invertible with inverse
  \[
    \phi^{-1}(\cO^\la) =
    \sum_{\mu/\la \text{ short}} \cO^\mu \,,
  \]
  where the sum is over straight shapes $\mu \subset \cP_X$ such that $\la
  \subset \mu$ and $\mu/\la$ is a short skew shape. The coefficient of $\cO_\la$
  in $[\cO_X(1)] \cdot \cO_\nu$ is now given by
  \[
    \begin{split}
      \chi\left(\psi(\cO^\la) \cdot [\cO_X(1)] \cdot \cO_\nu\right)
      &= \chi\left([\cO_X(-1)] \cdot \phi^{-1}(\cO^\la) \cdot
         [\cO_X(1)] \cdot \cO_\nu\right) \\
      &= \chi\left(\phi^{-1}(\cO^\la) \cdot \cO_\nu\right) = C_\nu^\la(X) \,,
    \end{split}
  \]
  as required.
\end{proof}

\begin{remark}
  Although both the Pieri--Chevalley formula from \cite{fulton.lascoux:pieri,
  pittie.ram:pieri-chevalley} and the identities from
  \cite{buch.chaput.ea:chevalley} were proved for flag varieties defined over
  $\C$, they hold in arbitrary characteristic. One way to see this is to observe
  that the Graham-Willems formula \cite{graham:equivariant, willems:k-theorie}
  for the restrictions of equivariant $K$-theoretic Schubert classes to
  $T$-fixed points is independent of characteristic; this implies that the
  Schubert structure constants of the equivariant $K$-theory ring $K_T(G/P)$ are
  independent of characteristic. Indeed, Graham's proof of the formula uses only
  that the Bott--Samelson resolution of a Schubert variety is a rational
  resolution, which holds in any characteristic \cite{ramanathan:schubert}. We
  thank William Graham for a very helpful discussion about this.
\end{remark}

Given straight shapes $\la \subset \nu \subset \cP_X$, we set $\Delta
H^X_{\nu/\la}(r) = H^X_{\nu/\la}(r+1) - H^X_{\nu/\la}(r)$.

\begin{cor}\label{cor:hilbert}%
  The Hilbert polynomial $H^X_{\nu/\la}$ is determined by $H^X_{\nu/\la}(0) = 1$
  and
  \[
    \Delta H^X_{\nu/\la} = \sum_{\la \subset \mu \subsetneq \nu}
    C_\nu^\mu(X)\, H^X_{\mu/\la} \,,
  \]
  where the sum is over all straight shapes $\mu$ satisfying $\la \subset \mu
  \subsetneq \nu$.
\end{cor}
\begin{proof}
  Using that $[\cO_X(r)|_{X_\nu^\la}] = [\cO_X(1)]^r \cdot \cO_\nu \cdot
  \cO^\la$ holds in $K(X)$, we obtain from \Lemma{multO1} that
  $[\cO_{X_\nu^\la}(r+1)] = \sum_{\la \subset \mu \subset \nu} C_\nu^\mu(X)\,
  [\cO_{X_\mu^\la}(r)]$. The result follows by subtracting
  $[\cO_{X_\nu^\la}(r)]$ from both sides of this identity and applying $\chi$.
\end{proof}

\begin{cor}\label{cor:hilbert-pos}%
  The coefficients $a_d$ in the binomial expansion
  \[
    H^X_{\nu/\la}(r) = \sum_{d=0}^{|\nu/\la|} a_d \binom{r}{d}
  \]
  are positive integers.
\end{cor}
\begin{proof}
  We have $a_0 = H^X_{\nu/\la}(0) = 1$, and the remaining coefficients in
  \[
    \Delta H^X_{\nu/\la}(r) = \sum_{d=1}^{|\nu/\la|} a_d \binom{r}{d-1}
  \]
  are positive integers by \Corollary{hilbert} and induction on $|\nu/\la|$.
\end{proof}

\begin{example}
  The minuscule Richardson varieties defined by the shapes
  \[
    \theta_1 = \tableau{8}{{}&{}&{}&{}\\{}&{}}
    \text{ \ \ \ \ and \ \ \ \ }
    \theta_2 = \tableau{8}{&{}\\{}&{}&{}&{}\\{}}
  \]
  of types $A_6$ and $E_6$ have the same Hilbert polynomial
  \[
    H_{\theta_1}(r) = H_{\theta_2}(r) =
    1 + 11\,r + 43 \binom{r}{2} + 82 \binom{r}{3} +
    83 \binom{r}{4} + 43 \binom{r}{5} + 9 \binom{r}{6} \,.
  \]
  The Richardson varieties $X_{\theta_1}$ and $X_{\theta_2}$ are not isomorphic
  as algebraic varieties, for example because the singular locus of
  $X_{\theta_1}$ is isomorphic to $\bP^2$, whereas the singular locus of
  $X_{\theta_2}$ has two components isomorphic to $\bP^1$ and $\bP^1 \times
  \bP^2$.
\end{example}

\bibliographystyle{halpha}

\end{document}